\documentclass[10pt, reqno]{amsart}
\usepackage{lipsum}
\usepackage{a4wide}
\usepackage{amssymb}
\usepackage{amsmath}
\usepackage{dsfont}
\usepackage{amsthm}
\usepackage{tikz-cd}
\usepackage{amscd}
\usepackage{xcolor}
\usepackage{hyperref}
\usepackage{enumitem}
\usepackage{mathrsfs}
\hypersetup{colorlinks,linkcolor={blue},citecolor={blue},urlcolor={red}}
\usepackage{color}
\usepackage[all]{xy}
\theoremstyle{plain}
\usepackage[margin=2.4cm]{geometry}
\numberwithin{equation}{section}

\newtheorem{theorem}{Theorem}[section]

\newtheorem{lemma}[theorem]{Lemma}
\newtheorem{remark}[theorem]{Remark}
\newtheorem{proposition}[theorem]{Proposition}

\newtheorem{definition}[theorem]{Definition}

\title{Kazhdan isomorphism over families and distinction under close local fields}
\author{Sabyasachi Dhar}
\begin{document}
	\maketitle 
\begin{abstract}
Let $F$ and $F'$ be two non-Archimedean local fields with the same residue characteristic $p$. Kazhdan's isomorphism for split connected reductive groups over $\mathbb{Z}$, extended by Ganapathy to general connected reductive groups, compares complex Hecke algebras over $F$ and $F'$, provided they are sufficiently close. We show that the same characteristic-function correspondence defines an isomorphism of Hecke algebras with coefficients in a commutative ring $R$ with unity in which $p$ is invertible. The proof uses the integral structure constants of the fixed-level Hecke algebras. As a first application, we prove that Kazhdan transfer preserves integrality of irreducible
smooth $\overline{\mathbb{Q}}_\ell$-representations for $\ell\ne p$. As a second application, for corresponding quadratic
extensions $E/F$ and $E'/F'$ of odd residue characteristic, we construct a natural isomorphism
between the ${\rm GL}_n(F)$-coinvariants of a smooth $R[{\rm GL}_n(E)]$-module generated by its level-$m$ fixed vectors and the ${\rm GL}_n(F')$-coinvariants of its transfer. This isomorphism is compatible with change of coefficient ring and therefore preserves distinction defined by free coinvariants of rank one.
\end{abstract}

\section{Introduction}
This article aims to study Kazhdan's theory over arbitrary commutative rings and its compatibility with integrality and distinction. In the foundational work \cite{Kaz_split}, Kazhdan establishes a link between the representation theory for $\textbf{G}(F)$ and $\textbf{G}(F')$, where $\textbf{G}$ is a connected split reductive group defined over $\mathbb{Z}$, $F$ is a $p$-adic field, and $F'$ is a non-Archimedean local field of positive characteristic that is sufficiently close to $F$. Kazhdan's approach is via an isomorphism between Hecke algebras with complex coefficients, $\mathcal{H}(\textbf{G}(F),K)$ and $\mathcal{H}(\textbf{G}(F'),K')$, for some suitable choice of compact open pro-$p$ subgroups $K$ and $K'$ of $\textbf{G}(F)$ and $\textbf{G}(F')$, respectively. In subsequent works \cite{MR3432266, Kaz_generalreductive, Radhika_Pacific_Erratum}, Ganapathy extended Kazhdan's work to general connected reductive algebraic groups defined over non-Archimedean local fields, using the machinery of Bruhat--Tits theory. 

Since the early 1990s, the study of representations with coefficients in complex numbers has been extended to representations with coefficients in fields -- and then rings -- of any characteristic. It is then natural to consider the question of extending Kazhdan's theory over arbitrary rings. Our first result proves Kazdhan isomorphism over an arbitrary commutative ring $R$ with unity. Let $G$ be a connected reductive algebraic group defined over a non-Archimedean local field $F$ with residue characteristic $p$. Assume that $p\in R^\times$. Fix a parahoric model $\mathscr{K}$ attached to a special vertex in the Bruhat--Tits building of $G(F)$, and put
$$ K_m:={\rm ker}(\mathscr{K}(\mathfrak{o}_F)\xrightarrow {{\rm mod}-\mathfrak{p}_F^m} \mathscr{K}(\mathfrak{o}_F/\mathfrak{p}_F^m)),\,\, m\geq 1, $$
where $\mathfrak{o}_F$ is the ring of integers of $F$, and $\mathfrak{p}_F$ is the unique maximal ideal of $\mathfrak{p}_F$. Let $F'$ be another non-Archimedean local field that is $m$-close to $F$, i.e., there exists a ring isomorphism $\mathfrak{o}_F/\mathfrak{p}_F^m\rightarrow \mathfrak{o}_{F'}/\mathfrak{p}_{F'}^m$. Let $G'/F'$, $\mathscr{K}'$, and $K_m'$ be the corresponding objects over $F'$ (see $\S$ \ref{par_close_local} for instance). Let $\mathcal{H}_R(G(F),K_m)$ be the Hecke algebra consisting of locally constant, compactly supported, bi-$K_m$-invariant, $R$-valued functions on $G(F)$, with convolution normalized by ${\rm vol}(K_m)=1$.
\begin{theorem}\label{main_intro}
Let $F$ and $F'$ be two non-Archimedean local fields with the same residue characteristic $p$. Let $R$ be a commutative ring with identity, and assume that $p$ is invertible in $R$. Let $G/F$ and $G'/F'$ be two connected reductive groups as above. Then there exists an integer $N>m$ such that if $F$ and $F'$ are $N$-close, we have an $R$-algebra isomorphism 
$$ {\rm Kaz}_{m,R} : \mathcal{H}_R(G(F),K_m)\xrightarrow{\sim} \mathcal{H}_R(G'(F'),K_m'). $$ 
\end{theorem}
When $R=\mathbb{C}$, the above theorem follows from \cite[Theorem 4.1]{Kaz_generalreductive}. When $G$ splits over $\mathbb{Z}$, the map ${\rm Kaz}_{m,\mathbb{C}}$ coincides with the Kazhdan's construction in \cite[Theorem A]{Kaz_split}. The key ingredient in their proofs is the fact that the Hecke algebra $\mathcal{H}_\mathbb{C}(G(F),K_m)$ is a finitely generated $\mathbb{C}$-algebra, which follows from the fundamental work of \cite{Bernstein_center}. When $R$ is a Noetherian $\mathbb{Z}_\ell$-algebra, it follows from a recent seminal work of Dat--Helm--Kurinczuk--Moss (\cite{Hecke_finiteness}) that $\mathcal{H}_R(G(F),K_m)$ is a finitely generated $R$-algebra. Then, following the similar arguments of Kazhdan and Ganapathy, one can deduce Theorem \ref{main_intro} for Noetherian $\mathbb{Z}_\ell$-algebras. In this article, we give a simpler proof of Theorem \ref{main_intro}, which avoids the finite-generation of $\mathcal{H}_R(G(F),K_m)$, and at the same time, our argument also works for any commutative ring with identity in which $p$ is invertible. The idea is as follows. Recall that the Hecke algebra $\mathcal{H}_R(G(F),K_m)$ is defined via the
convolution using a Haar measure on $G(F)$, and we have
\begin{enumerate}
\item  $\mathcal{H}_{\mathbb{Z}}(G(F),K_m)$ is a free $\mathbb{Z}$-module with basis given by the characteristic functions on the $K_m$-double cosets in $G(F)$.
\item the product of any two such indicator functions is a $\mathbb{Z}$-linear combination of other such indicator functions, where the coefficients depend only on the structure of the double cosets of $K_m$ in $G(F)$ (\cite[Subsection 1.3]{Cartier_p_adic_groups}).
\end{enumerate} 
Thus, $\mathcal{H}_R(G(F),K_m)$ is obtained as the base change from the $\mathbb{Z}$-module $\mathcal{H}_\mathbb{Z}(G(F),K_m)$ defined by ($1$) and ($2$) above. The bijection between the following set of double cosets
$$ K_m\backslash G(F)/K_m \rightarrow K_m'\backslash G'(F')/K_m' $$
induces a $\mathbb{Z}$-module isomorphism
$$  {\rm Kaz}_{m,\mathbb{Z}} : \mathcal{H}_\mathbb{Z}(G(F),K_m)\rightarrow \mathcal{H}_\mathbb{Z}(G'(F'),K_m'). $$
Since $\mathcal{H}_\mathbb{Z}(G(F),K_m)$ is a subring of $\mathcal{H}_\mathbb{C}(G(F),K_m)$, checking that the above $\mathbb{Z}$-module isomorphism ${\rm Kaz}_{m,\mathbb{Z}}$ preserves multiplication can be done after passing to
$\mathcal{H}_\mathbb{C}(G(F),K_m)$, where we can apply the results of \cite{Kaz_split, Kaz_generalreductive}. Tensoring over $\mathbb{Z}$ gives the required $R$-algebra isomorphism ${\rm Kaz}_{m,R}$.

The first application of Theorem \ref{main_intro} concerns integral $\ell$-adic representations with $\ell\ne p$. Fix an algebraic closure $\overline{\mathbb{Q}}_\ell$ of $\mathbb{Q}_\ell$, and let $\overline{\mathbb{Z}}_\ell$ be the ring of integers of $\overline{\mathbb{Q}}_\ell$. A smooth $\ell$-adic representation $G(F)$ is called integral if it is of finite length and admits a $\overline{\mathbb{Z}}_\ell[G(F)]$-stable lattice. Let $\pi$ be  an irreducible smooth $\ell$-adic representation of $G(F)$ such that $\pi^{K_m}\ne\{0\}$. Then the space $V^{K_m}$ is a simple $\mathcal{H}_{\overline{\mathbb{Q}}_\ell}(G(F),K_m)$-module (\cite[Chapter I, Section 6.3]{Vigneras_modl_book}). In fact, we have the following equivalence of categories
$$ {\rm Rep}_{\overline{\mathbb{Q}}_\ell}(G(F),K_m) \rightarrow {\rm Mod}-\mathcal{H}_{\overline{\mathbb{Q}}_\ell}(G(F),K_m) $$
$$ V\mapsto V^{K_m}, $$
where ${\rm Rep}_{\overline{\mathbb{Q}}_\ell}(G(F),K_m)$ is the category of smooth $\ell$-adic representations generated by its $K_m$-fixed vectors, and ${\rm Mod}-\mathcal{H}_{\overline{\mathbb{Q}}_\ell}(G(F),K_m)$ is the category of $\mathcal{H}_{\overline{\mathbb{Q}}_\ell}(G(F),K_m)$-modules. Using Theorem \ref{main_intro} for $R=\overline{\mathbb{Q}}_\ell$, we get an irreducible smooth $\ell$-adic representation $(\pi',V')$ of $G'(F')$ such that $V'^{K_m'}\ne \{0\}$. Then we have the following result.
\begin{theorem}\label{main_intro_2}
Let $F$, $F'$, $G$ and $G'$ be as in Theorem \ref{main_intro}. Let $(\pi,V)$ be an irreducible smooth $\ell$-adic representation of $G(F)$ with $V^{K_m}\ne\{0\}$, and let $(\pi',V')$ be its Kazhdan transfer. Then $(\pi,V)$ is integral if and only if $(\pi',V')$ is integral.
\end{theorem}
This gives another proof of the preservation of integrality under Kazhdan's transfer noted in the proof of \cite[Theorem 9.10]{li2024close}. The argument transfers a Hecke-stable lattice (see $\S$ \ref{integral_rep} for the definition) in the finite-dimensional $K_m$-fixed vector space and then constructs a stable lattice in the whole representation.

As a second application of Theorem \ref{main_intro}, we prove the compatibility of ${\rm GL}_n(F)$-distinction of smooth $R[{\rm GL}_n(E)]$-modules, where $E$ is a quadratic field extension of $F$, and $R$ is a commutative ring with identity such that $p\in R^\times$ and $p\ne 2$. Over an algebraically closed field of characteristic different from $p$, a smooth representation $\pi$ of ${\rm GL}_n(E)$ is said to be ${\rm GL}_n(F)$-distinguished if 
$$ {\rm Hom}_{{\rm GL}_n(F)}(\pi,\mathds{1})\ne \{0\}, $$ 
where $\mathds{1}$ stands for the trivial representation. Equivalently, the space of ${\rm GL}_n(F)$-coinvariants $\pi_{{\rm GL}_n(F)}$ is non-trivial. Multiplicity-one implies that this coinvariant is one-dimensional; see \cite[Theorem 4.1]{Vincent_GLn_dist} and the references there. Distinction for the pair $({\rm GL}_n(E), {\rm GL}_n(F))$ and its relation to Asai $L$-factors have been studied in \cite{Flicker_BSMF_1, Flicker_Asai, Dipendra_distinguished, Anand_Kable_Asai, Anand_IMRN, Anand_IMRN_Asai, Nadir_IMRN, Vincent_GLn_dist, Nadir_Kurinczuk_mod_l_Asai, mod_l_distinguished}. In families, \cite{dhar2026asai}
establishes a necessary condition involving Asai gamma factors for Whittaker distinction of cuspidal Whittaker type modules over the coefficient rings considered there.

For an arbitrary commutative ring with identity $R$, a smooth $R[{\rm GL}_n(E)]$-module $V$ is called ${\rm GL}_n(F)$-distinguished if $V_{{\rm GL}_n(F)}$ is a free $R$-module of rank one. This is the definition used throughout; over general rings, it is stronger than the mere existence of a non-zero ${\rm GL}_n(F)$-invariant $R$-linear form. Let $K_m(E)=I_n+\mathfrak{p}_E^mM_n(\mathfrak{o}_E)$, where $I_n$ is the $n\times n$ identity matrix. For these subgroups, \cite[Section 2.2 \& Theorem 3.1]{Solleveld_Meyer_BT} shows that ${\rm Rep}_{R}({\rm GL}_n(E), K_m(E))$ is a Serre subcategory and hence is closed under submodules. This gives an equivalence of categories
$$ V\mapsto V^{K_m(E)},\,\,\, {\rm Rep}_R({\rm GL}_n(E),K_m(E)) \xrightarrow{\sim} {\rm Mod}-\mathcal{H}_R({\rm GL}_n(E),K_m(E)). $$
By Theorem \ref{main_intro}, the above equivalence yields an equivalence of categories between ${\rm Rep}_R({\rm GL}_n(E),K_m(E))$ and ${\rm Rep}_R({\rm GL}_n(E'),K_m(E'))$.
\begin{theorem}\label{main_intro_3}
Let $N\geq m+1$ be two positive integers. Suppose $F$ and $F'$ are $N$-close and have odd residue characteristic $p$. Let $E$ and $E'$ be quadratic field extensions of $F$ and $F'$, respectively. Let $M\in {\rm Rep}_R({\rm GL}_n(E),K_m(E))$, and let $M'\in {\rm Rep}_R({\rm GL}_n(E'),K_m(E'))$ be its Kazhdan transfer. Then there exists an isomorphism 
$$ M_{{\rm GL}_n(F)}\simeq M'_{{\rm GL}_n(F')}, $$
compatible with change of coefficient ring $R\rightarrow S$. In particular, $M$ is ${\rm GL}_n(F)$-distinguished if and only if $M'$ is ${\rm GL}_n(F')$-distinguished.
\end{theorem}
The isomorphism in Theorem \ref{main_intro_3} is constructed from an isomorphism of the permutation module on ${\rm GL}_n(F)\backslash{\rm GL}_n(E)/K_m(E)$ and its counterpart over $F'$, compatible with their Hecke actions. Hilbert 90 and a fixed-point lemma for pro-$p$ groups (see Lemma \ref{key_lemma_1}) identify the orbit sets with fixed double cosets. The Kazhdan isomorphism at level $m+1$ then proves compatibility with a set of generators of Hecke algebras. This provides a way to transfer the fixed-level distinction questions between local non-Archimedean fields of characteristic zero and positive characteristic. Extensions of this method to other algebraic groups require analogous category and cohomology properties, together with the corresponding equivariant Hecke-module comparison.  
 
The paper is organized as follows. Section 2 contains some preliminaries on Hecke algebras. Section 3 reviews the comparison of reductive algebraic groups and parahoric group schemes over close local fields. In Section 4, we prove Theorem \ref{main_intro}. In Section 5, we set up some initial results and deduce Theorem \ref{main_intro_2}. In Section 6, we first discuss some standard facts about distinction and then prove Theorem \ref{main_intro_3}.
\vspace{0.5 cm}\\
\textbf{Acknowledgments.} The author would like to thank Radhika Ganapathy for the comments on the initial draft of this paper. The author thanks Dipendra Prasad and U.K. Anandavardhanan for their interest in this work. The author is grateful to the Indian Institute of Technology Bombay for supporting this research through the Institute Postdoctoral Fellowship.
\vspace{0.5 cm}\\
\textbf{Declaration on AI usage.} The author has used ChatGPT for completing the proofs of Lemma \ref{key_lemma} and Lemma \ref{key_lemma_1}, besides writing of the paper and locating references. The author has independently verified the materials obtained with the assistance of ChatGPT and takes full responsibility for the correctness of the paper.

\section{Hecke algebras}
In this section, we recall the notion of Hecke algebras. All rings are assumed to be commutative with identity. For a precise reference, see \cite[Chapter 1, Section 3]{Vigneras_modl_book}.

\subsection{}
Let $F$ be a non-Archimedean local field of residue characteristic $p$. Let $R$ be a commutative ring with identity, and assume that $p\in R^\times$. Let $G$ be the $F$-points of a connected reductive algebraic group defined over $F$. For an open compact subgroup $K$ of $G$, let $\mathcal{H}_R(G,K)$ be the space of locally constant, compactly supported functions $f:G\rightarrow R$ such that 
$$ f(k_1gk_2) = f(g), $$
for all $k_1,k_2\in K,\, g\in G$. Assume that $K$ is of pro-order invertible in $R$. By \cite[Chapter I, 2.4]{Vigneras_modl_book}, there exists a left-invariant Haar measure $\mu$ on $G$ such that $\mu(K)$ is invertible in $R$, and we get a multiplicative structure $*$ on the module $\mathcal{H}_R(G,K)$, given by convolution
\begin{equation}\label{conv_prod}
(f_1* f_2)(x) = \int_{G} f_1(g)\,f_2(g^{-1}x)\,d\mu(g),
\end{equation}
for all $f_1,f_2\in \mathcal{H}_R(G,K)$. With this convolution product, $\mathcal{H}_R(G,K)$ becomes an associative $R$-algebra. 

\subsection{}
Let $\{g_\alpha: \alpha\in \Lambda\}$ be a set of representatives for the double cosets of $G$ by $K$. Then $G$ is the disjoint union of $Kg_\alpha K$. For any index $\alpha\in \Lambda$, let $u_\alpha$ be defined by
$$ u_\alpha = \mu(K)^{-1} t_{g_\alpha}, $$
where $t_{g_\alpha}$ denotes the characteristic function on the double coset $Kg_{\alpha}K$. The family $\{u_\alpha:\alpha\in \Lambda\}$ is a basis of the free $R$-module $\mathcal{H}_R(G,K)$. Moreover, for $\alpha,\beta\in \Lambda$, the convolution product of $u_\alpha$ and $u_\beta$ is given by
\begin{equation}\label{basis_multiplication}
 u_\alpha * u_\beta = \sum_{\gamma\in \Lambda} c_{\alpha\beta\gamma} u_\gamma, 
\end{equation}
where the coefficients $c_{\alpha\beta\gamma}$ are integers depending on the structure of the $K$-double cosets in $G$ (\cite[Subsection 1.3]{Cartier_p_adic_groups}).

If $K'$ is a compact open subgroup of $K$, then $\mathcal{H}_R(G,K)$ is an sub-$R$-algebra of $\mathcal{H}_R(G,K')$. Let $\mathcal{H}_R(G)$ be the union 
$$ \bigcup_{K\in \Omega(G)} \mathcal{H}_R(G,K), $$
where $\Omega(G)$ is the set of open compact subgroups of $G$ of pro-order invertible in $R$. Then $\mathcal{H}_R(G)$ consists of locally constant and compactly supported functions on $G$. For the convolution product defined by (\ref{conv_prod}), $\mathcal{H}_R(G)$ is an associative $R$-algebra. Note that $\mathcal{H}_R(G)$ is commutative if and only if $G$ is commutative. The space $\mathcal{H}_R(G)$ is called the Hecke algebra of $G$, and $\mathcal{H}_R(G,K)$ is called the Hecke algebra of $G$ relative to $K$ (or simply the Hecke algebra whenever $G$ and $K$ are understood from the context). 

\section{Parahoric group schemes}
In this section, we review parahoric group schemes. We also set up some notation and recall results on compatibility of parahoric group schemes under close local fields. We follow the notations of \cite{cong_parahoric_radhika} and \cite{Kaz_generalreductive}.
\subsection{}
Let $F$ be a non-Archimedean local field with ring of integers $\mathfrak{o}_F$ and the unique maximal ideal $\mathfrak{p}_F$. We write $k_F = \mathfrak{o}_F/\mathfrak{p}_F$ for the finite residue field with characteristic $p$. Let $G$ be a connected reductive group defined over $F$. In this article, we always identify a group scheme over $F$ with its $F$-rational points. Let $\mathcal{B}(G,F)$ be the Bruhat--Tits building of $G(F)$ (see \cite{Bruhat_Tits_II}, \cite{Kaletha_Prasad} for details). Let $T$ be a maximal $F$-split torus in $G$. Let $\mathcal{A}(T,F)$ be the apartment in $\mathcal{B}(G,F)$ corresponding to the $F$-split torus $T$. Given a facet $\mathcal{F}$ in $\mathcal{A}(T,F)$, there exists a smooth affine $\mathfrak{o}_F$-group scheme $\mathscr{P}_{\mathcal{F}}$, which is called a {\it parahoric group scheme} associated with $\mathcal{F}$. For each integer $n\geq 1$, consider the following subgroup
$$ K_n(F)={\rm Ker}\big(\mathscr{P}_{\mathcal{F}}(\mathfrak{o}_F)
\xrightarrow{{\rm mod}-\mathfrak{p}_F^n}\mathscr{P}_{\mathcal{F}}
(\mathfrak{o}_F/\mathfrak{p}_F^n)\big). $$
These congruence subgroups yield a filtration of the parahoric subgroup $K$ of $G(F)$ underlying the group scheme $\mathscr{P}_{\mathcal{F}}$.



\subsection{Reductive groups under close local fields}
Fix a positive integer $m$. Let $F$ and $F'$ be two non-Archimedean local fields that are $m$-close, i.e., there exists a ring isomorphism $\mathfrak{o}_F/\mathfrak{p}_F^m
\xrightarrow{\sim}\mathfrak{o}_{F'}/\mathfrak{p}_{F'}^m$. We fix a ring isomorphism $\lambda_F:\mathfrak{o}_F/\mathfrak{p}_F^m
\rightarrow\mathfrak{o}_{F'}/\mathfrak{p}_{F'}^m$ such that 
$$ \lambda_F(\varpi_F+\mathfrak{p}_{F}^m) = 
\varpi_F'+\mathfrak{p}_{F'}^m, $$
where $\varpi_F$ (resp. $\varpi_F'$) is the uniformizer of $F$ (resp. $F'$). For an object $X$ over $F$, we denote by $X'$ the corresponding object over $F'$.

\subsubsection{}\label{red_close_local}
Let $G$ be a connected reductive algebraic group defined over $F$. We need to make sense of what it means to give an algebraic group $G'$ over $F'$. We first explain this for quasisplit groups. The construction is due to \cite[$\S$3]{cong_parahoric_radhika}. So, assume $G$ is quasisplit. Let $(G_0,T_0,B_0,\{u_\alpha\}_{\alpha\in\Delta})$ be a pinned, split, connected, reductive group defined over $\mathbb{Z}$ with based root datum $(R,\Delta)$, where $\{u_\alpha\}_{\alpha\in\Delta}$ is a splitting as in \cite[$\S$3.2.2]{Bruhat_Tits_II}. Fix a separable closure $F_s$ of $F$. An $F$-form of $G_0$ is an algebraic group $H$ over $F$ such that $H\times_F F_s$ is isomorphic to $G_0\times_{\mathbb{Z}} F_s$. 

We know that the $F$-isomorphism classes of quasisplit groups that are $F$-forms of $G_0$ are paremetrized by the pointed cohomology set $H^1(\Gamma_F,{\rm Aut}(R,\Delta))$, where $\Gamma_F$ is the absolute Galois group ${\rm Gal}(F_s/F)$. Let $E_{qs}(G_0,F)_m$ be the set of $F$-isomorphism classes of quasisplit groups $G$ that split and become isomorphic to $G_0$, over an at most $m$-ramified extension of $F$. \cite[Lemma 3.1]{cong_parahoric_radhika} gives a bijection 
$$ E_{qs}(G_0,F)_m\rightarrow H^1(\Gamma_F/I_F^m,{\rm Aut}(R,\Delta)). $$
Since $F$ and $F'$ are $m$-close, using the Deligne isomorphism, there exists a bijection of pointed cohomology sets (\cite[Lemma 3.3]{cong_parahoric_radhika})
$$ H^1(\Gamma_F/I_F^m,{\rm Aut}(R,\Delta))\rightarrow H^1(\Gamma_{F'}/I_{F'}^m,{\rm Aut}(R,\Delta)), $$ 
which induces a bijection
$$ E_{qs}(G_0,F)_m \rightarrow E_{qs}(G_0,F')_m $$
$$ [G]\mapsto [G']. $$
Moreover, with the cocycles chosen compatibly, this will yield data $(G,T,B)$ over $F$, where $T$ is a maximal $F$-torus in $G$, $B$ is a Borel subgroup of $G$ containing $T$ and defined over $F$. Similarly, we have the data $(G',T',B')$ over $F'$.
\subsubsection{}
We now consider the case where $G$ is a connected reductive group defined over $F$, not necessarily quasisplit. Recall that $G$ is an inner form of a quasisplit group over $F$, say $G^*$, whose isomorphism class is parametrized by the cohomology set $H^1\big({\rm Gal}(\widetilde{F}/F),G^*_{\rm ad}(\widetilde{F}\big)$, where $\widetilde{F}$ is the completion of the maximal unramified extension of $F$ and $G^*_{\rm ad}$ is the adjoint group of $G^*$. Let $G'^*$ be the quasisplit connected reductive group over $F'$ corresponding to $G^*$ as above. Then we have a bijection (\cite[Lemma 5.1]{cong_parahoric_radhika}) between the cohomology sets
$$ H^1\big({\rm Gal}(\widetilde{F}/F),G^*_{\rm ad}(\widetilde{F})\big)\xrightarrow{\sim} H^1\big({\rm Gal}(\widetilde{F'}/F'),G'^*_{\rm ad}(\widetilde{F'})\big). $$
Thus, we get a connected reductive group $G'/F'$ corresponding to $G/F$. Moreover, using the above isomorphism at the level of cocycles, we get a triple $(G,S,A)$ induced from the triple $(G^*,T^*,B^*)$, where $S$ is a maximal $\widetilde{F}$-split $F$-torus of $G$ containing a maximal $F$-split torus $A$ in $G$. Similarly, we have the triple $(G',S',A')$ over $F'$ induced from $(G'^*,T'^*,B'^*)$. 

\subsection{Parahoric group scheme under close local fields}\label{par_close_local}
Let $(G,S,A)$ and $(G',S',A')$ be the triples as above. Let $v$ be a special vertex in the apartment $\mathcal{A}(A,F)$. Let $\mathscr{K}$ be the smooth affine parahoric group scheme over $\mathfrak{o}_F$ corresponding to $v$, and let $K=\mathscr{K}(\mathfrak{o}_F)$. For $m\geq 1$, we have the following congruence subgroups
$$ K_m(F) = {\rm Ker}\big(\mathscr{K}(\mathfrak{o}_F)
\xrightarrow{{\rm mod}-\mathfrak{p}_F^m}\mathscr{K}
(\mathfrak{o}_F/\mathfrak{p}_F^m)\big). $$
Since $F$ and $F'$ are $m$-close, there exists a simplicial isomorphism $\mathscr{A}_m:\mathcal{A}(A,F)\rightarrow\mathcal{A}(A',F')$ (\cite[Subsection 4C]{cong_parahoric_radhika}). Let $v'$ be the special vertex in the apartment $\mathcal{A}(A',F')$--which is the image of $v$ under the map $\mathscr{A}_m$. Let $K'$ be the maximal parahoric subgroup of $G'(F')$ corresponding to $v'$. Let $\mathscr{K}'$ be the smooth affine group scheme over $\mathfrak{o}_{F'}$ underlying $K'$, and $K'_m$ the $m$-th congruence subgroup as defined above. In \cite[Corollary 6.3]{cong_parahoric_radhika}, the author proves the following result.
\begin{proposition}\label{par_con} 
Given an integer $m\geq 1$, there exists an integer $r\geq m$ such that if $F$ and $F'$ are $r$-close, then the parahoric group schemes $\mathscr{K}\times_{\mathfrak{o}_F}
\mathfrak{o}_F/\mathfrak{p}_F^m$ and $\mathscr{K}'\times_{\mathfrak{o}_{F'}}
\mathfrak{o}_{F'}/\mathfrak{p}_{F'}^m$ are isomorphic. In particular, we have a group isomorphism
$$ \mathscr{K}(\mathfrak{o}_F/
\mathfrak{p}_F^m)\simeq\mathscr{K}'
(\mathfrak{o}_{F'}/\mathfrak{p}_{F'}^m). $$
\end{proposition}

\section{Kazhdan Isomorphism over families}\label{Kaz_fam}
In this section, we prove Theorem \ref{main_intro}. In \cite[Theorem 4.1]{Kaz_generalreductive}, Ganapathy establishes an isomorphism between the Hecke algebras
$$ \mathcal{H}_{\mathbb{C}}(G(F),K_m(F))\xrightarrow{\sim} \mathcal{H}_{\mathbb{C}}(G'(F'),K_m(F')), $$ 
which is a generalization of \cite[Theorem A]{Kaz_split} to general connected reductive groups defined over non-Archimedean local fields. Let $R$ be a commutative ring with identity such that $p\in R^\times$. The construction of the above isomorphism as $\mathbb{C}$-vector spaces can be extended to the coefficient ring $R$--which we recall in the following. Then we show that this $R$-module isomorphism is in fact an $R$-algebra isomorphism.
\subsection{}
Consider the triple $(G,S,A)$, where $G$ is a connected reductive algebraic group defined over $F$, $S$ is a maximal $F$-torus in $G$ that splits over $\widetilde{F}$, and $A$ is a maximal $F$-split torus contained in $S$. Let $\Phi(G,A)$ be the root system of $G$ with respect to $A$. Let $M$ be the centralizer of $A$ in $G$. Let $\Omega_M$ be the Iwahori-Weyl group of $M$ with $\Omega_M^+$ denoting the set of dominant elements in $\Omega_M$, i.e.,
$$\Omega_M^+:=\{\tau\in\Omega_M:\langle\alpha,\tau\rangle\geq 0 ,\,\forall\alpha\in\Delta_0\},$$
where $\Delta_0$ is the set of simple roots in $\Phi(G,A)$. Let $\widetilde{p}:\Omega_M\rightarrow M(F), \tau\mapsto n_\tau$ be the section of Kottwitz homomorphism $\kappa_{M,F}$ (\cite[Proposition 3.3]{Kaz_generalreductive}). For each element $\tau\in\Omega_M^+$, the double coset $G_\tau(F)=K(F)n_\tau K(F)$ is a homogeneous space under the action of $K(F)\times K(F)$, defined by
$$ (a,b)g:=agb^{-1}, $$ 
for all $a,b\in K$ and $g\in G_\tau(F)$. Then the following set of double cosets
$$ X_\tau=\big\{K_m(F)gK_m(F):g\in G_\tau(F)\big\} $$ 
is a homogeneous space under the action of the finite group $K(F)/K_m(F)\times K(F)/K_m(F)$. Let $\Gamma_\tau$ be the stabilizer of $K_m(F)\,n_\tau\,K_m(F)$. Let $R$ be a commutative ring with identity in which $p$ is invertible. We fix a Haar measure $\mu$ on $G(F)$ such that $\mu(K_m(F)) =1$. Let $\mathcal{H}_R(G(F),K_m(F))$ be the Hecke algebra with coefficients in $R$, i.e., it consists of locally constant, compactly supported, bi-$K_m(F)$-invariant $R$-valued functions on $G(F)$, where the multiplicative structure on $\mathcal{H}_R(G(F),K_m(F))$ is given by the convolution product
$$ (f_1*f_2)(g) = \int_{G(F)} f_1(x)f_2(x^{-1}g)\,d\mu(g), $$  for $f_1,f_2\in\mathcal{H}_R(G(F),K_m(F))$. For each $g\in G(F)$, let $t_g$ be the characteristic function on the double coset $K_m(F)\,g\,K_m(F)$. For a set of representatives $\{g_\alpha\}_\alpha$ for the double cosets of $G(F)$ modulo $K_m(F)$, the set $\{t_{g_\alpha}\}_\alpha$ is a basis for the $\mathbb{Z}$-module $\mathcal{H}_\mathbb{Z}(G(F),K_m(F))$, where the convolution products of two basis elements are given by (\ref{basis_multiplication}). Using the canonical ring homomorphism $\varphi_R:\mathbb{Z}\rightarrow R$, we have
$$ \mathcal{H}_R(G(F),K_m(F)) = \mathcal{H}_\mathbb{Z}(G(F),K_m(F))
\otimes_{\mathbb{Z},\varphi_R} R. $$

\subsection{}
Let $\mathscr{M}$ (resp. $\mathscr{M}'$) be the smooth affine $\mathfrak{o}_F$-group scheme underlying $M$ (resp. $M'$). Consider the following subgroups
$$ M_m(F) = {\rm Ker}\big(\mathscr{M}(\mathfrak{o}_F)\rightarrow \mathscr{M}(\mathfrak{o}_F/\mathfrak{p}_F^m)\big) $$
and
$$ M_m'(F') = {\rm Ker}\big(\mathscr{M'}(\mathfrak{o}_{F'})\rightarrow \mathscr{M'}(\mathfrak{o}_{F'}/\mathfrak{p}_{F'}^m)\big). $$
Let $h$ be the integer defined as in \cite[Section 8]{Chai_Kang}. Assume that the fields $F$ and $F'$ are $r$-close, where $r\geq m+4h$. Then there exists a group isomorphism (\cite[Proposition 2.4]{Radhika_Pacific_Erratum})
\begin{equation}\label{minimal_levi_con}
M(F)/M_m(F)\xrightarrow{\sim} M'(F')/M_{m}'(F').
\end{equation}
Let $\widetilde{p}':\Omega_{M'}\rightarrow M(F'), \tau'\mapsto n_{\tau'}$ be the section of the Kottwitz homomorphism $\kappa_{M',F'}$. Then, under the isomorphism (\ref{minimal_levi_con}), the class of
$n_\tau$ is mapped to the class of $n_{\tau'}$. 
\subsubsection{}
Let $W(G,A)$ (resp. $W(G',A')$) denote the Weyl group of $G$ (resp. $G'$) with respect to the maximal torus $A$ (resp. $A'$). Then there are the following isomorphisms (\cite[Proposition 4.3]{Kaz_generalreductive}):
$$ W(G,A)\backslash\Omega_M\simeq W(G',A')\backslash\Omega_{M'}\,\,\,\,\text{and}\,\,\, \Omega_M\simeq\Omega_{M'}, $$ 
and the set of dominant weights $\Omega_M^+$ is mapped onto the set of dominant weights $\Omega_{M'}^+$ under the above isomorphism. Using Proposition \ref{par_con}, we get a group isomorphism
$$ p_m:K(F)/K_m(F)\times K(F)/K_m(F)\longrightarrow 
K(F')/K_m(F')\times K(F')/K_m(F') $$ such that $p_m(\Gamma_\tau)=\Gamma_{\tau'}$ for all $\tau\in \Omega_M^+$. Let $T_\tau(F)\subseteq K\times K$ be a set of representatives of
$$ \Big(K(F)/K_m(F)\times K(F)/K_m(F)\Big)/\Gamma_\tau. $$
Similarly, define $T_{\tau'}(F')$. Then we have a bijection $T_\tau(F)\rightarrow T_{\tau'}(F')$,\, $(k_i,k_j)\mapsto (k_i',k_j')$, which gives a $\mathbb{Z}$-module isomorphism
$$ {\rm Kaz}_{m,\mathbb{Z}} : \mathcal{H}_{\mathbb{Z}}(G(F),K_m(F))\longrightarrow 
\mathcal{H}_{\mathbb{Z}}(G'(F'),K_m(F')) $$
by requiring that
$$ t_{k_in_\tau k_j^{-1}}\longmapsto 
t_{k_i'n_{\tau'}k_j'^{-1}}, $$
for all $\tau\in\Omega_M^+$ and $(k_i,k_j)\in T_\tau(F)$. 
Then the algebra morphism $\varphi_R : \mathbb{Z}\rightarrow R$ induces an $R$-module isomorphism
\begin{equation}\label{Kaz_isom_module}
{\rm Kaz}_{m,R} : \mathcal{H}_R(G(F),K_m(F))
\rightarrow \mathcal{H}_R(G'(F'),K_m(F')).
\end{equation} 
With these, we now prove the following theorem.
\begin{theorem}\label{Kaz_isom}
Let $p$ be a prime and $R$ be a commutative ring with identity in which $p$ is invertible.  Let $m\geq 1$ and $r\geq m$ be as in \cite[Proposition 2.4]{Radhika_Pacific_Erratum}. Then there exists an integer $N\geq r$ such that for any two non-Archimedean $N$-close local fields $F$ and $F'$ with residue characteristic $p$, the map ${\rm Kaz}_{m,R}$ is an isomorphism of $R$-algebras.
\end{theorem}
\begin{proof}
The only thing to check is that the map ${\rm Kaz}_{m,R}$ preserves multiplication. Consider the $\mathbb{Z}$-module isomorphism 
$$ {\rm Kaz}_{m,\mathbb{Z}} : \mathcal{H}_{\mathbb{Z}}(G(F),K_m(F))\longrightarrow 
\mathcal{H}_{\mathbb{Z}}(G'(F'),K_m(F')) $$
$$  t_{k_in_\tau k_j^{-1}}\longmapsto 
t_{k_i'n_{\tau'}k_j'^{-1}}. $$
Note that $\mathcal{H}_\mathbb{Z}(G(F),K_m(F))$ is a subring of $\mathcal{H}_\mathbb{C}(G(F),K_m(F))$, and we have 
$$ {\rm Kaz}_{m,\mathbb{Z}} = {\rm Kaz}_{m,\mathbb{C}}|_{\mathcal{H}_\mathbb{Z}(G(F),K_m(F))}. $$
Using \cite[Theorem 4.1]{Kaz_generalreductive}, there exists an integer $N\geq r$ such that if $F$ and $F'$ are $N$-close, the map ${\rm Kaz}_{m,\mathbb{C}}$ is an isomorphism of $\mathbb{C}$-algebras. Since the convolution product of the basis elements $t_{k_in_\tau k_j^{-1}}$ (see (\ref{basis_multiplication})) depends only on the $K_m(F)$-double cosets in $G(F)$ and 
$$ {\rm Kaz}_{m,\mathbb{C}} (f_1*f_2) = {\rm Kaz}_{m,\mathbb{C}}(f_1)*{\rm Kaz}_{m,\mathbb{C}}(f_2) $$
for all $f_1,f_2\in \mathcal{H}_\mathbb{Z}(G(F),K_m(F))$, the map ${\rm Kaz}_{m,\mathbb{Z}}$ preserves the multiplication, and hence is a $\mathbb{Z}$-algebra isomorphism. Now, taking base change from $\mathbb{Z}$ to $R$ via the $\mathbb{Z}$-algebra homomorphism $\varphi_R:\mathbb{Z}\rightarrow R$, the $\mathbb{Z}$-algebra isomorphism ${\rm Kaz}_{m,\mathbb{Z}}$ induces an $R$-algebra isomorphism
$$ {\rm Kaz}_{m,R} : \mathcal{H}_{R}(G(F),K_m(F))\rightarrow 
\mathcal{H}_{R}(G'(F'),K_m(F')). $$
Hence, the theorem.
\end{proof}

\subsection{The ${\rm GL}_n$ case:} In this subsection, we briefly discuss the construction of Kazhdan isomorphism for $G={\rm GL}_n$, which will be used later. This is a particular case of the construction of the Kazhdan map (\ref{Kaz_isom_module}). For a reference, see \cite{Kaz_split}. Let $K_0$ be the compact open subgroup ${\rm GL}_n(\mathfrak{o}_F)$, and let $X_\ast(T)_{+}=\{\lambda=(\lambda_1,\dots,\lambda_n)\in\mathbb{Z}^n: \lambda_1\geq \lambda_2\geq\dots\geq \lambda_n\}$. Then we have the Cartan decomposition:
$$ G(F) =\coprod_{\lambda\in X_\ast(T)_{+}} K_0\varpi_\lambda K_0, $$
where $\varpi_\lambda={\rm diag}(\varpi^{\lambda_1},\dots,\varpi^{\lambda_n})$. For each $\lambda$, the set $\{K_mxK_m:x\in K_0\varpi_\lambda K_0\}$ is a homogeneous space for the action of the finite group $G(\mathfrak{o}_F/\mathfrak{p}_F^m)\times G(\mathfrak{o}_F/\mathfrak{p}_F^m)$. Let $\Gamma_\lambda$ denote the stabilizer of $K_m\varpi_\lambda K_m$, and let $T_\lambda\subseteq K_0\times K_0$ be a set of representatives of the quotient group $(G(\mathfrak{o}_F/\mathfrak{p}_F^m)\times G(\mathfrak{o}_F/\mathfrak{p}_F^m))/\Gamma_\lambda$. 

Fix $N\geq m$. Let $F'$ be a local non-Archimedean local field that is $N$-close to $F$. In fact, one can take $N=m$ (see \cite[Corollary 3.15]{MR3432266}). Then we have a similarly defined set $T_\lambda'\subseteq K_0'\times K_0'$ and a bijection $T_\lambda\xrightarrow{\sim} T_\lambda'$. The set-theoretic bijection
\begin{equation}\label{set_bij}
\beta_m: K_m(F)\backslash G(F)/K_m(F)\xrightarrow{\sim} K_m(F')\backslash G(F')/K_m(F')
\end{equation}
$$ K_m(F)\,a_i\varpi_\lambda a_j^{-1} K_m(F)\mapsto K_m(F')\,a_i'\varpi'_\lambda a_j'^{-1} K_m(F') $$
where $(a_i',a_j')\in T_\lambda'$ is the image of $(a_i,a_j)\in T_\lambda$ under the bijection $T_\lambda\rightarrow T_\lambda'$. The map $\beta_m$ induces an isomorphism of $R$-algebras
$$ {\rm Kaz}_{m,R}: \mathcal{H}_{R}(G(F),K_m(F))\xrightarrow{\sim} \mathcal{H}_{R}(G(F'),K_m(F')) $$
$$ t_{x}\mapsto t_{\beta_m(x)}. $$

\section{Integrality under close local fields}
In this section, we prove that integrality is preserved under close local fields. This is a consequence of Theorem \ref{Kaz_isom}. 
\subsection{}
Let $F$ be a non-Archimedean local field of residue characteristic $p$. Let $G$ be the $F$-rational points of some connected reductive algebraic group defined over $F$. Let $R$ be a commutative ring with unity such that $p\in R^\times$. In this paper, all $R[G]$-modules are assumed to be smooth, i.e., the stabilizer of each element of the underlying module is an open subgroup of $G$. Let $M$ be an $R[G]$-module. For any compact open pro-$p$ subgroup $K$ of $G$, the $K$-fixed vectors $M^K$ is a module over $\mathcal{H}_R(G,K)$. Let ${\rm Irr}_R(G,K)$ be the set of isomorphism classes of simple $R[G]$-modules $M$ such that $M^K \ne \{0\}$. We denote by $\mathcal{G}_R(G,K)$ the set of isomorphism classes of simple $\mathcal{H}_R(G,K)$-modules. Then we have a bijection (see \cite[Chapter I, Section 6.3]{Vigneras_modl_book})
\begin{equation}\label{modules_bij_1}
\mathcal{J} : {\rm Irr}_R(G,K) \xrightarrow{\sim} \mathcal{G}_R(G,K),\,\,\, V\mapsto V^K.
\end{equation}
Let $F'$ be a non-Archimedean local field that is $m$-close to $F$. Let $G'$ be the $F'$-rational points of some connected reductive group defined over $F'$, associated with $G$ (see $\S$ \ref{red_close_local}). Let $K_m$ and $K_m'$ be the $m$-th congruence subgroups of $G$ and $G'$, respectively (see $\S$ \ref{par_close_local} for the definition). Theorem \ref{Kaz_isom} gives a Hecke algebra isomorphism 
\begin{equation}\label{Kaz_isom_1}
{\rm Kaz}_{m,R} : \mathcal{H}_R(G,K_m) \xrightarrow{\sim} \mathcal{H}_R(G',K_m')
\end{equation}
provided that the fields $F$ and $F'$ are $N$-close, where $N\geq m$. 
\begin{lemma}\label{modules_bij_2}
There exists a bijective correspondence between the sets
${\rm Irr}_R(G,K_m)$ and ${\rm Irr}_R(G',K_m')$.
\end{lemma}
\begin{proof}
Consider the diagram 
$$
\xymatrix{
	{\rm Irr}_R(G,K_m) \ar[dd]_{\mathcal{J}}  \ar[rr]^{}  &&  {\rm Irr}_R(G',K_m') \ar[dd]^{\mathcal{J}'} \\\\
	\mathcal{G}_R(G,K_m) \ar[rr]_{\Phi} &&  \mathcal{G}_R(G',K_m')
}
$$
where the horizontal map $\Phi$ is the bijection, induced by ${\rm Kaz}_{m,R}$. Let $M\in {\rm Irr}_R(G,K_m)$. The $K_m$-fixed vectors $M^{K_m}$ is a simple $\mathcal{H}_R(G,K_m)$-module by (\ref{modules_bij_1}). Using Kazhdan isomorphism (\ref{Kaz_isom_1}), $M^{K_m}$ is also a simple $\mathcal{H}_R(G',K_m')$-module. Consider
$$ V'=R[G'/K_m']\otimes_{\mathcal{H}_R(G',K_m')} M^{K_m}. $$
Then $V'^{K_m'}\simeq M^{K_m}$ as $\mathcal{H}_R(G',K_m')$-modules. Note that a proper $R[G']$-submodule of $V'$ does not have non-trivial $K_m'$-fixed vectors. Let $V'_0$ be the sum of these submodules. Since $e_{K_m'}$ commutes with sums, we have $V_0'^{K_m'}=\{0\}$. Thus, $V_0'$ is proper and contains all proper $R[G']$-submodules of $V'$. Then the quotient $M'=V'/V_0'$ is a simple $R[G']$-module, and 
$$ \mathcal{J}'(M') \simeq M'^{K_m'}$$
Thus, the composition $\mathcal{J}'^{-1}\circ \Phi\circ \mathcal{J}$ gives a bijection 
$$ {\rm Irr}_R(G,K_m)\xrightarrow{\sim}{\rm Irr}_R(G',K_m') $$
$$ M\mapsto M' $$
which makes the above diagram commutative.
\end{proof}

\subsection{Integral representation}\label{integral_rep}
We fix an algebraic closure $\overline{\mathbb{Q}}_\ell$ of $\mathbb{Q}_\ell$, where $\ell\ne p$. Let $\overline{\mathbb{Z}}_\ell$ be the ring of integers of $\overline{\mathbb{Q}}_\ell$. For any field extension $L$ of $\mathbb{Q}_\ell$, we denote by $\mathfrak{o}_L$ the ring of integers of $L$. Note that $\mathfrak{o}_L=L\cap \overline{\mathbb{Z}}_\ell$. A smooth $\overline{\mathbb{Q}}_\ell$-representation $(\pi,V)$ of $G$ is realizable over $L$ if there exists a $G$-stable $L$-vector subspace $V_L\subseteq V$ such that $V_L\otimes_{L}\overline{\mathbb{Q}}_\ell = V$. In \cite[II.4.7]{Vigneras_modl_book}, Vign\'eras proved that a smooth irreducible $\overline{\mathbb{Q}}_\ell$-representation of $G$ is realizable over a finite extension of $\mathbb{Q}_\ell$. A smooth $\overline{\mathbb{Q}}_\ell$-representation $(\pi,V)$ of $G$ is called {\it integral} if it is of finite length and there exists a $\overline{\mathbb{Z}}_\ell[G]$-stable lattice $\mathcal{L}$ in $V$. Suppose $(\pi,V)$ is realizable over $L$, and let $V_L$ be an $L$-structure on $V$. Then there exists an $\mathfrak{o}_L[G]$-stable lattice $\mathcal{L}_0$ in $V_L$ such that $\mathcal{L}_0\otimes_{\mathfrak{o}_L}\overline{\mathbb{Z}}_\ell$ is a $\overline{\mathbb{Z}}_\ell[G]$-stable lattice in $V$. 

\subsubsection{}
Let $L$ be a finite extension of $\mathbb{Q}_\ell$. Let $(\pi,V)$ be a smooth admissible $L$-representation of $G$. Let $K$ be a compact open pro-$p$ subgroup of $G$. Then $V^K$ is a finite-dimensional $L$-vector space. Let $M$ be an $\mathfrak{o}_L$-submodule of $V^K$. 
\begin{definition}
$M$ is called an $\mathfrak{o}_L$-lattice stable under $\mathcal{H}_{\mathfrak{o}_L}(G,K)$ if the following holds:
\begin{enumerate}
    \item $M$ is finitely generated as an $\mathfrak{o}_L$-module;
    \item $M$ spans $V^K$ over $L$, i.e., $L\cdot M=V^K$;
    \item for all $f\in \mathcal{H}_{\mathfrak{o}_L}(G,K)$ and $m\in M$, we have $f* m\in M$.
\end{enumerate}    
\end{definition}
We now prove the following key lemma.
\begin{lemma}\label{key_lemma}
Let $L$ be a finite extension of $\mathbb{Q}_\ell$. Let $(\pi,V)$ be an irreducible smooth $L$-representation of $G$. Let $K$ be a compact open pro-$p$ subgroup of $G$ with $V^K\ne \{0\}$. Let $M\subseteq V^K$ be an $\mathfrak{o}_L$-lattice stable under $\mathcal{H}_{\mathfrak{o}_L}(G,K)$. Then 
$$ \mathcal{L}:=\mathfrak{o}_L[G]M $$
is an $\mathfrak{o}_{L}[G]$-stable lattice in $V$.
\end{lemma}
\begin{proof}
Normalize the Haar measure $\mu$ on $G$ such that $\mu(K)=1$. Let $g\in G$. Recall that $t_g$ stands for the characteristic function on the double coset $KgK$. For all $m\in M$, we have the following identity
$$ \dfrac{1}{N_g}\,(t_g * m) = (e_K\circ \pi(g)) (m), $$
where $N_g=[K:K\cap gKg^{-1}]$. Note that $N_g$ is a power of $p$ as $K$ is a pro-$p$ group. Therefore, $N_g\in \mathfrak{o}_L^\times$. Since $t_g$ belongs to $\mathcal{H}_{\mathfrak{o}_L}(G,K)$ and $\mathcal{H}_{\mathfrak{o}_L}(G,K)M\subseteq M$, the operator $e_K\circ \pi(g)$ stabilizes $M$. Now, for $v=\sum_{i}a_i\pi(g_i)m_i\in\mathcal{L}$ and $x\in G$, we have
$$ (e_K\circ \pi(x))(v) = \sum_{i}(e_K\circ \pi(xg_i))(m_i). $$
Since each summand in the RHS belongs to $M$, we have $(e_K\circ \pi(x))(\mathcal{L})\subseteq M$ for all $x\in G$. In particular, $e_K(\mathcal{L}) = M$. The non-zero subspace $\mathcal{L}\otimes_{\mathfrak{o}_L} L$ is $G$-stable, so the irreducibility of $(\pi,V)$ gives $\mathcal{L}\otimes_{\mathfrak{o}_L} L=V$. Set
$$ V_0=\bigcap_{g\in G}{\rm ker}(e_K\circ\pi(g):V\rightarrow V^K). $$
Note that $V_0$ is $G$-stable. It is proper because $e_K(m) =m$ for all $m\in M\setminus\{0\}$. Consequently, $V_0=\{0\}$ by the irreducibility of $(\pi,V)$. 

Fix another compact open pro-$p$ subgroup $J$ of $G$. Since $V^J$ is finite-dimensional and
$$ \bigcap_{g\in G}{\rm ker}((e_K\circ \pi(g))|_{V^{J}}) = \{0\}, $$
there exist finitely many elements $g_1,\dots,g_r$ in $G$ such that $(e_K\circ \pi(g_j))|_{V^{J}}=0$ for all $j=1,2,\dots r$. This implies that the mapping 
$$ T: V^{J}\rightarrow (V^K)^{\oplus r};\,v\mapsto (e_K(\pi(g_1)v),\dots,e_K(\pi(g_r)v)) $$ is injective. Since $(e_K\circ \pi(g_j))(\mathcal{L})\subseteq M$ for all $j$, we have $T(\mathcal{L}^{J})\subseteq (M^K)^{\oplus r}$. Note that $M^{\oplus r}$ is a finitely generated module over the Noetherian discrete valuation ring $\mathfrak{o}_L$. Therefore, its submodule $T(\mathcal{L}^{J})$ is also finitely generated. Thus, $\mathcal{L}^{J}$ is finitely generated and torsion-free, and hence is free of finite rank over $\mathfrak{o}_L$. Note that the averaging operator $e_J$ preserves $\mathcal{L}$. Thus, 
$$ L\otimes_{\mathfrak{o}_L} \mathcal{L}^{J}=
L\otimes_{\mathfrak{o}_L}(e_{J}(\mathcal{L}))=
e_{J}(L\otimes_{\mathfrak{o}_L} \mathcal{L}) = V^{J}. $$
This shows that $\mathcal{L}^{J}$ is a lattice in $V^{J}$. 

Choose a decreasing sequence $K_1\supset K_2\supset\cdots\supset K_n\supset\cdots$ of compact open pro-$p$ subgroups. By the smoothness of $(\pi,V)$, we have
$$ \mathcal{L}=\bigcup_n \mathcal{L}^{K_n}. $$
Note that $\mathcal{L}^{K_n}\subseteq \mathcal{L}^{K_{n+1}}$ for all $n$. Consider the map $A_n=e_{K_n}|_{V^{K_{n+1}}}: V^{K_{n+1
}}\rightarrow V^{K_n}$. Also note that $A_n$ preserves $\mathcal{L}$, i.e., we have a map $A_n:\mathcal{L}^{K_{n+1}}\rightarrow \mathcal{L}^{K_n}$. Then 
$$ \mathcal{L}^{K_{n+1}}=\mathcal{L}^{K_n}\oplus {\rm ker}(A_n),\,\,\,\forall n\geq 1. $$
Since each $\mathcal{L}^{K_n}$ is a free module of finite rank over the principal ideal domain $\mathfrak{o}_L$, the submodule ${\rm ker}(A_n)$ is also a free $\mathfrak{o}_L$-module of finite rank. Extending a basis at every step and taking their union, we obtain an $\mathfrak{o}_L$-basis of $\mathcal{L}$, and it is an $L$-basis of $V$ because $L\otimes_{\mathfrak{o}_L} \mathcal{L}=V$. Hence, $\mathcal{L}$ is a $G$-stable lattice in $V$. This completes the proof.
\end{proof}

\subsubsection{}
Fix a positive integer $m$. Let $F$ and $F'$ be two non-Archimedean local fields of the same residue characteristic $p$. Assume that $F$ and $F'$ are $m$-close, i.e., there exists a ring isomorphism
$$ \mathfrak{o}_F/\mathfrak{p}_F^m \xrightarrow{\sim} \mathfrak{o}_{F'}/\mathfrak{p}_{F'}^m. $$ 
Let $\textbf{G}$ be a connected reductive algebraic group defined over $F$. Let $\textbf{G}'$ be the connected reductive group defined over $F'$ associated with $\textbf{G}$. Set $G=\textbf{G}(F)$ and $G'=\textbf{G}'(F')$. We consider the $m$-th congruence subgroups $K_m$ and $K_m'$, defined by
$$ K_m = {\rm Ker}\big(\mathscr{K}(\mathfrak{o}_F)\xrightarrow{{\rm mod}-\mathfrak{p}_F^m} \mathscr{K}(\mathfrak{o}_F/\mathfrak{p}_F^m)\big) $$
and 
$$ K_m' = {\rm Ker}\big(\mathscr{K}'(\mathfrak{o}_{F'})\xrightarrow{{\rm mod}-\mathfrak{p}_{F'}^m} \mathscr{K}'(\mathfrak{o}_{F'}/\mathfrak{p}_{F'}^m)\big), $$
where $\mathscr{K}$ and $\mathscr{K}'$ are the parahoric group schemes associated with $\textbf{G}$ and $\textbf{G}'$, respectively.
Now, applying Lemma \ref{modules_bij_2} for $R=\overline{\mathbb{Q}}_\ell$, we get a bijection 
\begin{equation}\label{Kaz_transfer}
\mathfrak{L} : {\rm Irr}_{\overline{\mathbb{Q}}_\ell}(G,K_m) \xrightarrow{\sim} {\rm Irr}_{\overline{\mathbb{Q}}_\ell}(G',K_m')
\end{equation}
$$ V\mapsto V'. $$
With these, we now prove Theorem \ref{main_intro_2}.
\begin{theorem}\label{Integral_close_fields}
Under the hypothesis of Theorem \ref{Kaz_isom}, the Kazhdan transfer (\ref{Kaz_transfer}) preserves integrality of smooth irreducible $\overline{\mathbb{Q}}_\ell$-representations, for $\ell\ne p$.
\end{theorem}
\begin{proof}
Let $V$ be an irreducible smooth $\overline{\mathbb{Q}}_\ell$-representation of $G$. Set $V'=\mathfrak{L}(V)$. Choose a finite extension $L/\mathbb{Q}_\ell$ such that both $V$ and $V'$ are realizable over $L$, and there is a Hecke module isomorphism
$$ j: V_L^{K_m}\xrightarrow{\sim} V_L'^{K_m'}, $$
where $V_L$ (resp. $V'_L$) is the $L$-structure on $V$ (resp. $V'$). Suppose $V$ is integral. Choose a $G$-stable $\mathfrak{o}_L$-lattice $\mathcal{L}\subseteq V_L$, expanding the coefficient field $L$ if necessary. By admissibility of $V$, the fixed vectors $\mathcal{L}^{K_m}$ is an $\mathfrak{o}_{L}$-lattice in $V_L^{K_m}$ (i.e., $\mathcal{L}^{K_m}$ is finitely generated as an $\mathfrak{o}_L$-module and spans $V_L^{K_m}$ over $L$). Note that the integral Hecke algebra $\mathcal{H}_{\mathfrak{o}_L}(G,K_m)$ preserves the lattice $\mathcal{L}^{K_m}$.

Since the Kazhdan isomorphism (Theorem \ref{Kaz_isom}) is compatible with the change of coefficient rings, the isomorphism over $L$ is obtained from the isomorphism over $\mathfrak{o}_L$. Consequently
$$ j(\mathcal{L}^{K_m})\subseteq V_L'^{K_m'} $$
is an $\mathfrak{o}_L$-lattice stable under $\mathcal{H}_{\mathfrak{o}_L}(G',K_m')$. Applying Lemma \ref{key_lemma}, we get that
$$ \mathcal{L}'=\mathfrak{o}_L[G']
j(\mathcal{L}^{K_m}) $$
is a $G'$-stable lattice in $V_L'$. Then $\mathcal{L}'\otimes_{\mathfrak{o}_L} \overline{\mathbb{Z}}_\ell$ is a $G$-stable lattice in $V'$ proving the integrality of $V'$. The inverse Kazhdan isomorphism gives the reverse implication.
\end{proof} 

\section{Distinction under close local fields}
In this section, we discuss another application of Kazhdan isomorphism over families to the distinction problem. Let $G$ be a locally profinite group, and let $H$ be a closed subgroup of $G$. For any field $R$, an irreducible smooth $R$-representation $\pi$ of $G$ is called $H$-distinguished if there exists a non-zero $H$-invariant linear form on $\pi$. This is equivalent to saying that ${\rm Hom}_H(\pi,\mathds{1})$ is non-trivial, where $\mathds{1}$ stands for the trivial representation of $H$. Among many applications of the distinction pair, we focus on the situation where $G={\rm GL}_n(E)$ and $H={\rm GL}_n(F)$, where $E/F$ is a quadratic extension of non-Archimedean local fields of odd residue characteristic $p$.

Let $R$ be a commutative ring with identity such that $p$ is invertible in $R$. For a smooth $R[G]$-module $V$, we denote by $V_H$ the quotient module $V/V(H)$, where $V(H)$ is the sub-$R$-module generated by elements of the form $h\cdot v-v$ for all $h\in H$ and $v\in V$. When $R=\mathbb{C}$, $\overline{\mathbb{Q}}_\ell$, $\overline{\mathbb{F}}_\ell$, the multiplicity-one theorem (see \cite[Proposition 11]{Flicker_dist} for $R=\mathbb{C}$, $\overline{\mathbb{Q}}_\ell$ and \cite[Theorem 4.1]{Vincent_GLn_dist} for $R=\overline{\mathbb{F}}_\ell$) shows that an irreducible smooth $R$-representation $\pi$ of $G$ is $H$-distinguished if and only if the space $\pi_H$ is one-dimensional. 
\begin{definition}
A smooth $R[G]$-module $V$ is said to be $H$-distinguished (or distinguished by $H$) if $V_H$ is a free $R$-module of rank one.    
\end{definition}
\subsection{}
Let $K_m$ be the $m$-th congruence subgroup of $G$. Let $\mathscr{R}_R(G,K_m)$ denote the category of smooth $R[G]$-modules generated by their $K_m$-fixed vectors. We now explain a crucial property of $\mathscr{R}_R(G,K_m)$ that it is closed under subobjects. This will be used later. Let $\mathcal{B}(G)$ be the reduced Bruhat-Tits building of $G$. For a vertex $x=[\mathcal{L}]$ in $\mathcal{B}(G)$, choose $g\in G$ with $\mathcal{L}=[g\mathfrak{o}_E^n]$ and put
$$ K_{m,x}=gK_mg^{-1}\,\,\,\text{and}\,\,\, e_x=e_{K_{m,x}}, $$
where $e_{K_{m,x}}$ is the projection operator $V\rightarrow V^{K_{m,x}}$. The system $(K_{m,x},e_x)_{x\in\mathcal{B}(G)}$ is equivariant and is the congruence-idempotent system in \cite[Section 2.2]{Solleveld_Meyer_BT}, with their parameter $e=m-1$. By \cite[Theorem 3.1]{Solleveld_Meyer_BT}, we get a Serre subcategory of modules generated by the images of the operators $e_x$. Since $G$ acts transitively on the set of vertices in $\mathcal{B}(G)$, this Serre subcategory is precisely the category of $R[G]$-modules generated by their $K_m$-fixed vectors. This establishes the fact that $\mathscr{R}_R(G,K_m)$ is closed under sub-$R[G]$-modules.

Let $\mathcal{P}_R$ be the compact induction ${\rm ind}_{K_m}^G(\mathds{1}_R)$. Note that ${\rm ind}_{K_m}^G(\mathds{1}_R)=R[G/K_m]$. The module $\mathcal{P}_R$ is endowed with a left $G$-action and a right $\mathcal{H}_R(G,K_m)$-action by convolution. Since $K_m$ is a compact open pro-$p$ group, for any object $V\in\mathscr{R}_R(G,K_m)$, it follows from the Frobenius reciprocity that 
\begin{equation}\label{fixed_Km}
{\rm Hom}_G(\mathcal{P}_R,V)=V^{K_m}.
\end{equation}
We now prove a couple of lemmas.
\begin{lemma}
The module $\mathcal{P}_R$ is a projective generator in the category $\mathscr{R}_R(G,K_m)$.
\end{lemma}
\begin{proof}
We have to show that the functor $\mathcal{I}={\rm Hom}_G(\mathcal{P}_R,\_)$ is right exact and faithful. Let $V\xrightarrow{\varphi} W$ be a surjective morphism of smooth $R[G]$-modules. Let $w\in W^{K_m}$. Choose a lift $v$ of $w$ in $V$. By smoothness, there exists an open normal subgroup $J\subseteq K_m$ fixing $v$. Set
$$ \widetilde{v}=\frac{1}{[K_m:J]}\sum_{k\in K_m/J}k\cdot v. $$
Note that $\widetilde{v}$ is $K_m$-invariant and is mapped to $w$ under $\varphi$. Thus, we have a surjection $V^{K_m}\rightarrow W^{K_m}$, which implies that $\mathcal{I}(V)\rightarrow \mathcal{I}(W)$ is surjective by (\ref{fixed_Km}). Hence, $\mathcal{P}_R$ is projective.

Let $\varphi, \psi:V\rightarrow W$ be two morphisms of $R[G]$-modules such that $\mathcal{I}(\varphi)=\mathcal{I}(\psi)$. Then $\varphi$ and $\psi$ agree on $V^{K_m}$. Since $V$ is generated by $V^{K_m}$ as $R[G]$-modules, we have $\varphi=\psi$. Thus, $\mathcal{I}$ is a faithful functor. This completes the proof.
\end{proof}
\begin{lemma}
For any object $V\in\mathscr{R}_R(G,K_m)$, the map
$$ \Phi: \mathcal{P}_R\otimes_{\mathcal{H}_R(G,K_m)} V^{K_m}\rightarrow V $$
$$ 1_{gK_m}\otimes v\mapsto g\cdot v $$
is an isomorphism.  
\end{lemma}
\begin{proof}
Note that $\Phi$ is surjective as $V$ is generated by $V^{K_m}$ as $R[G]$-modules. For injectivity, we use the fact that $\mathscr{R}_R(G,K_m)$ is a Serre subcategory (\cite[Section 2.2 \& Theorem 3.1]{Solleveld_Meyer_BT}), and hence it is closed under subrepresentations. Put $\mathcal{N}={\rm ker}(\Phi)$. For convenience, let $\mathcal{H}_R$ be the Hecke algebra $\mathcal{H}_R(G,K_m)$. Then the map $f\mapsto 1_{K_m}\ast f$ gives an isomorphism of right $\mathcal{H}_R$-modules $\mathcal{H}_R\xrightarrow{\sim} \mathcal{P}_R^{K_m}$. We have the following decomposition
$$ \mathcal{P}_R =\mathcal{P}_R^{K_m}\oplus {\rm ker}(e_{K_m}) $$
as right $\mathcal{H}_R$-modules. Here, $e_{K_m}$ is the projection operator from $\mathcal{P}_R$ onto $\mathcal{P}_R^{K_m}$. Now, tensoring with $V^{K_m}$ gives
$$ \mathcal{P}_R\otimes_{\mathcal{H}_R} V^{K_m}
= (\mathcal{P}_R^{K_m}\otimes_{\mathcal{H}_R}V^{K_m}) \oplus ({\rm ker}(e_{K_m})\otimes_{\mathcal{H}_R}V^{K_m}). $$
Applying $e_{K_m}$ to the above identity, we have
$$ 
(\mathcal{P}_R\otimes_{\mathcal{H}_R}V^{K_m})^{K_m}\simeq \mathcal{P}_R^{K_m}\otimes_{\mathcal{H}_R}V^{K_m}\simeq \mathcal{H}_R\otimes_{\mathcal{H}_R}V^{K_m}\simeq V^{K_m}. $$
Under this identification, the restriction of $\Phi$ to $K_m$-fixed vectors is just the identity map. Consequently, $\mathcal{N}^{K_m} = \{0\}$. Note that $\mathcal{P}_R\otimes_{\mathcal{H}_R(G,K_m)}V^{K_m}\in\mathscr{R}_R(G,K_m)$. Since the category $\mathscr{R}_R(G,K_m)$ is closed under subrepresentations, we have $\mathcal{N}\in\mathscr{R}_R(G,K_m)$. Thus, $\mathcal{N}=\{0\}$, and this proves the injectivity of $\Phi$. Hence, the lemma follows.
\end{proof}
\begin{remark}\normalfont
The preceding two lemmas are valid for a general reductive group $\textbf{G}$ defined over $\mathbb{Z}$ and for a compact open pro-$p$ subgroup $K$ of $G=\textbf{G}(E)$, provided that the category $\mathscr{R}_R(G,K)$ is closed under submodules.
\end{remark}

\subsection{}
Next, we prove the following lemma on the fixed points of Hausdorff spaces under the action of pro-$p$-groups. This is crucial to establish a bijection between permutation modules on certain double coset spaces.
\begin{lemma}\label{key_lemma_1}
Let $K$ be a pro-$p$ group with $p\ne 2$. Let $X$ be a Hausdorff space, and assume that $K$ acts continuously and transitively on $X$. Let $\tau$ be an involution (i.e., a continuous group automorphism of order dividing $2$) on $K$, and let $\theta$ be an involution (i.e., a homeomorphism of order dividing $2$) on $X$ such that $\theta(k\cdot x) =\tau(k)\cdot \theta(x)$. Then
\begin{enumerate}
    \item $X^\theta$ is non-empty;
    \item $K^\tau$ acts transitively on $X^\theta$.
\end{enumerate}
\begin{proof}
Recall that every element $k$ in the pro-$p$ group $K$ admits a unique square root, denoted by $k^{1/2}$. We use the following standard facts:
\begin{itemize}
    \item The root $k^{1/2}$ belongs to $\overline{\langle k\rangle}$, where $\overline{\langle k\rangle}$ is the closure of the cyclic group $\langle k\rangle$ in $K$. Every closed subgroup containing $k$ contains its square root.
    \item Every continuous automorphism $\tau$ of $K$ preserves roots, i.e., $\tau(k^{1/2})=\tau(k)^{1/2}$.
\end{itemize}
Consider the semi-direct product $K' = K\rtimes \langle u\rangle$ with $u^2=1$ and $uku^{-1}=\tau(k)$, where $u$ acts on $X$ by $\theta$. Choose $x\in X$. Using the transitivity of $K$-action, there exists $k\in K$ such that $\theta(x)=k\cdot x$. Set $b=k^{-1}u$. Let $K'_x$ be the stabilizer subgroup of $x$ in $K'$. Then $b\in K'_x$ and $b^2\in K\cap K'_x$. Since $X$ is Hausdorff, the stabilizer $K'_x$ is closed. Note that 
$$ a=(b^2)^{1/2}\in \overline{\langle b^2\rangle}\subseteq K\cap K'_x, $$ 
and $ab=ba$. Thus, the element $t=ba^{-1}$ fixes $x$ and satisfies $t^2=1$.

Set $s=tu^{-1}$. The equation $t^2=1$ gives $\tau(s)=s^{-1}$. Since $\tau$ preserves its root, we have $\tau(s^{1/2})=s^{-1/2}$. Then 
$$ s^{1/2}us^{-1/2} = s^{1/2}\tau(s^{-1/2})u = su =t. $$
Since $t$ fixes $x$, the element $u$ fixes $s^{-1/2}\cdot x$. Therefore
$$ \theta(s^{-1/2}\cdot x) = s^{-1/2}\cdot x. $$
This shows that $X^\theta$ is non-empty. For proving transitivity, let $x,y\in X^\theta$. Since $K$ acts transitively on $X$, there exists $k\in K$ such that $k\cdot x=y$. Then
$$ k\cdot x=y=\theta(y)=\theta(k\cdot x) = \tau(k)\cdot x. $$
This implies that the element $c=k^{-1}\tau(k)\in K_x$. Note that $c^{1/2}\in K_x$ and $\tau(c^{1/2})=c^{-1/2}$. Then 
$$ \tau(kc^{1/2}) =\tau(k)c^{-1/2}=kcc^{-1/2} = kc^{1/2}. $$
Thus, $kc^{1/2}\in K^\tau$, and $(kc^{1/2})\cdot x = y$. Hence, the lemma follows.
\end{proof}
\end{lemma}
Let $\sigma$ be the non-trivial element of ${\rm Gal}(E/F)$. Let $\iota$ be the anti-involution on $G$, given by $\iota(g)=\sigma(g)^{-1}$ for all $g\in G$. For any $\iota$-stable subset $X\subseteq G$, we denote by $X^{\iota}$ the set of fixed points under $\iota$. Note that $G^{\iota}=\{x\in G\mid x\sigma(x)=1\}$. Using Hilbert 90, we have a bijection 
\begin{equation}\label{coset_bij_1}
H\backslash G \xrightarrow{\sim} G^{\iota};\,\,\,\,\,\, Hg\mapsto g^{-1}\sigma(g).
\end{equation}
Consider the map $\phi_m: H\backslash G/K_m\rightarrow (K_m\backslash G/K_m)^\iota$, defined by
$$ \phi_m(HgK_m):=K_m\,g^{-1}\sigma(g)K_m. $$
Since $K_m$ is $\sigma$-stable and $H=G^\sigma$, the map $\phi_m$ is well-defined.
\begin{theorem}\label{bij_coset_map}
The map $\phi_m$ is a bijection.
\end{theorem}
\begin{proof}
For convenience, set $\eta(g)=g^{-1}\sigma(g)$. Let $X$ be an $\iota$-stable double coset in $K_m\backslash G/K_m$. Take $U=K_m\times K_m$. Consider the action of $U$ on $X$, defined by 
$$ (a,b)\cdot x = axb^{-1}, $$
for $(a,b)\in U$ and $x\in X$. Note that $X$ is a Hausdorff space and that the $U$-action on $X$ is continuous and transitive. Define $\tau((a,b))=(\sigma(b),\sigma(a))$. Then $\tau$ is an involution on $U$. Take $\theta=\iota|_{X}$. Then 
$$ \iota((a,b)\cdot x)=\iota(axb^{-1})= \sigma(b)\sigma(x)^{-1}\sigma(a)^{-1}=((\sigma(b),\sigma(a))\cdot\iota(x) =\tau((a,b))\cdot \iota(x). $$
Applying Lemma \ref{key_lemma_1}, we get that $X^{\iota}\ne\emptyset$. Then there exists $x\in X$ such that $x\sigma(x)=1$. In particular, $x\in G^{\iota}$. By (\ref{coset_bij_1}), there exists $g\in G$ (unique up to left multiplication by $H$) such that $\eta(g)=x$. Therefore
$$ X=K_mxK_m=K_m\eta(g)K_m = \phi_m(HgK_m). $$
This shows that $\phi_m$ is surjective. To prove injectivity, let $g_1, g_2\in G$ be such that
$$ K_m\eta(g_1)K_m=K_m\eta(g_2)K_m. $$
Note that $U^{\tau}=\{(k,\sigma(k)):k\in K_m\}$. Since $\eta(g_1)$ and $\eta(g_2)$ are $\iota$-fixed, and $U^{\tau}$ acts transitively on $X^{\iota}$ (by Lemma \ref{key_lemma_1}), there exists $k\in K_m$ such that
$$ \eta(g_1)=k^{-1}\eta(g_2)\sigma(k)=\eta(g_2k).
$$
By (\ref{coset_bij_1}), we have $g_1=hg_2k$, for some $h\in H$. Hence, $Hg_1K_m=Hg_2K_m$. Thus, $\phi_m$ is injective. This completes the proof.
\end{proof}
\begin{remark}\normalfont
Theorem \ref{bij_coset_map} is valid for any triple $(\textbf{G},K,p)$, where $\textbf{G}$ is a connected reductive group defined over $F$ having the property ${\rm H}^1(\sigma,\textbf{G}(E))=\{1\}$, $K$ is a compact open pro-$p$ subgroup of $\textbf{G}(E)$ such that $\sigma(K)=K$, and $p$ is an odd prime. 
\end{remark}

\subsection{}
Fix $m\geq 1$ and $N\geq m+1$. Let $F$ and $F'$ be two $N$-close non-Archimedean local fields of residue characteristic $p$ , and assume that $p\ne 2$. Let $E$ be a quadratic field extension of $F$. Then there exists a quadratic field extension $E'/F'$ such that $E'$ is $eN$-close to $E$, where $e=1$ or $2$ (see \cite[Lemma 4.6]{Dhar_2024} for a proof). We fix two isomorphisms
$$ \Lambda_F:\mathfrak{o}_F/\mathfrak{p}_F^N\xrightarrow{\sim}\mathfrak{o}_{F'}/\mathfrak{p}_{F'}^N,\,\,\,\,\,\Lambda_E:\mathfrak{o}_E/\mathfrak{p}_E^{eN}\xrightarrow{\sim}\mathfrak{o}_{E'}/\mathfrak{p}_{E'}^{eN} $$
which makes the following diagram
$$ \begin{tikzcd}
\mathfrak{o}_F/\mathfrak{p}_F^N \arrow{r}{\Lambda_F} \arrow[swap]{d}{} & \mathfrak{o}_{F'}/\mathfrak{p}_{F'}^N \arrow{d}{} \\
\mathfrak{o}_E/\mathfrak{p}_E^{eN} \arrow{r}{\Lambda_E} & \mathfrak{o}_{E'}/\mathfrak{p}_{E'}^{eN}
\end{tikzcd}
$$
commutative. Here, the vertical arrows are the natural inclusion maps. We choose a uniformizer $\varpi_F$ (resp. $\varpi_F'$) of $F$ (resp. $F'$) such that $\Lambda_F(\varpi_F+\mathfrak{p}_F^N)=\varpi_F'+\mathfrak{p}_{F'}^N$. Fix the Galois automorphisms $\sigma\in {\rm Gal}(E/F)$ and $\sigma'\in{\rm Gal}(E'/F')$ such that $\Lambda_E\circ \sigma=\sigma'\circ\Lambda_E$ (see \cite[Remark 4.7 \& 4.8]{Dhar_2024}). Choose uniformizers $\varpi_E$ and $\varpi_E'$ of the fields $E$ and $E'$, respectively, such that $\Lambda_E(\varpi_E+\mathfrak{p}_E^{eN})=\varpi_{E'}+\mathfrak{p}_{E'}^{eN}$, and 
$$ \sigma(\varpi_E) =\varepsilon \varpi_E,\,\,\, \sigma'(\varpi_{E'})=\varepsilon \varpi_{E'}, $$ where $\varepsilon =1$ in the unramified case and $\varepsilon=-1$ in the ramified case. For convenience, write
$$ G={\rm GL}_n(E),\,\, H={\rm GL}_n(F),\,\, G'={\rm GL}_n(E'),\,\, H'={\rm GL}_n(F'). $$
Recall that we have the Kazhdan isomorphism
\begin{equation}\label{Kaz_isom_3}
{\rm Kaz}_{m,R}: \mathcal{H}_R(G,K_m)\xrightarrow{\sim} \mathcal{H}_R(G',K_m').
\end{equation}
Since $\mathscr{R}_R(G,K_m)$ is a Serre subcategory, we have a bijection 
$$ \mathscr{R}_R(G,K_m)\xrightarrow{\sim} {\rm Mod}-\mathcal{H}_R(G,K_m) $$
$$ V\mapsto V^{K_m}, $$
where ${\rm Mod}-\mathcal{H}_R(G,K_m)$ is the category of $\mathcal{H}_R(G,K_m)$-modules. Then, using (\ref{Kaz_isom_3}), we get the following bijection
\begin{equation}\label{Serre_cat_close_fields}
\mathscr{R}_R(G,K_m)\xrightarrow{\sim} \mathscr{R}_R(G',K_m').
\end{equation}
With these, we now deduce the following result.
\begin{theorem}
Let $m\geq 1$ and $N\geq m+1$. Let $F$ and $F'$ be two $N$-close non-Archimedean local fields of odd residue characteristic $p$. Let $E$ and $E'$ be the corresponding quadratic field extensions, satisfying the compatibility conditions as above. Let $R$ be a commutative ring with identity such that $p\in R^\times$. Let $V\in\mathscr{R}_R(G,K_m)$, and let $V'\in \mathscr{R}_R(G',K_m')$ be its Kazhdan transfer by (\ref{Serre_cat_close_fields}). Then there is a natural isomorphism
$$ V_H \simeq V'_{H'} $$
compatible with change of the coefficient ring $R\rightarrow S$. In particular, $V$ is $H$-distinguished if and only if $V'$ is $H'$-distinguished.
\end{theorem}
\begin{proof}
Consider the compactly induced module $\mathcal{P}_R=R[G/K_m]$ and the Hecke algebra $\mathcal{H}_R=\mathcal{H}_R(G,K_m)$. Recall that $\mathcal{P}_R$ is endowed with the left $G$-action and the right $\mathcal{H}_R$-action by convolution. Note that
$$ (\mathcal{P}_R)_H = R[H\backslash G/K_m]. $$
This is a right $\mathcal{H}_R$-module. Since $V\simeq \mathcal{P}_R\otimes_{\mathcal{H}_R} V^{K_m}$, we have 
\begin{equation}\label{coinv_1}
V_H \simeq R[H\backslash G/K_m]\otimes_{\mathcal{H}_R} V^{K_m}
\end{equation}
Similarly, we have the isomorphism 
\begin{equation}\label{coinv_2}
V'_{H'}\simeq R[H'\backslash G'/K_m']\otimes_{\mathcal{H}'_{R}} V'^{K_m'},
\end{equation}
where $\mathcal{H}_R'=\mathcal{H}_{R}(G',K_m')$. The main step is to construct an isomorphism 
$$ R[H\backslash G/K_m]\simeq R[H'\backslash G'/K_m'] $$
which is compatible with the right action of the corresponding Hecke algebras. Recall that we have a set-theoretic bijection (see (\ref{set_bij}))
$$ \beta_m : K_m\backslash G/K_m\rightarrow K_m'\backslash G'/K_m' $$
induces the Kazhdan isomorphism ${\rm Kaz}_{m,R}$. Note that $\beta_{m}$ is compatible with the anti-involutions $\iota$ and $\iota'$, which induce a bijection on the level of $\iota$ and $\iota'$-fixed points:
$$ \beta_m: (K_m\backslash G/K_m)^{\iota}\xrightarrow{\sim} (K_m'\backslash G'/K_m')^{\iota}. $$
Thus, we get a bijection
$$ \alpha_m: H\backslash G/K_m\rightarrow H'\backslash G'/K_m', $$
given by $\alpha_m = \phi_m'^{-1}\circ \beta_m\circ \phi_m$, where $\phi_m$ (resp. $\phi_m'$) is the bijection $H\backslash G/K_m\rightarrow (K_m\backslash G/K_m)^{\iota}$ (resp. $H'\backslash G'/K_m'\rightarrow (K_m'\backslash G'/K_m')^{\iota})$ (see Theorem \ref{bij_coset_map}). The map $\alpha_m$ gives a $\mathbb{Z}$-module isomorphism 
\begin{equation}\label{relative_module_isom}
\widetilde{\alpha}_{m,\mathbb{Z}}:\mathbb{Z}[H\backslash G/K_m]\xrightarrow{\sim} \mathbb{Z}[H'\backslash G'/K_m']
\end{equation}
$$ 1_x \mapsto 1_{\alpha_m(x)}. $$
Write $\mathcal{C}=\mathbb{Z}[1/p]$. After changing the coefficient ring $\mathbb{Z}\rightarrow \mathcal{C}$, we now show that $\widetilde{\alpha}_{m,\mathcal{C}}$ respects the Hecke algebra actions. For any element $x\in G$, we have the following decomposition:
$$ K_m x K_m = \coprod_{j=1}^d k_jx K_m $$
where $k_j$ runs over a set of representatives of the quotient space $K_m/J_x$ with $J_x=K_m\cap xK_m x^{-1}$ and $d=[K_m: K_m\cap xK_mx^{-1}]$. Recall that $t_x$ denotes the characteristic function $1_{K_m x K_m}$. Then the right $\mathcal{H}_\mathcal{C}$-action on $\mathcal{C}[H\backslash G/K_m]$ gives
\begin{equation}\label{Hecke_action}
1_{HgK_m}\cdot t_x = \sum_{j=1}^d 1_{Hgk_jxK_m}
\end{equation}
Consider the elements $x_i={\rm diag}(\varpi_EI_i, I_{n-i}), 1\leq i\leq n$. Note that the Hecke algebra $\mathcal{H}_R$ is generated by $t_k$ ($k\in K_0/K_m$), $t_{x_i}$ ($1\leq i\leq n$), and $t_{x_n^{-1}}$. Also note that $x_i^{-1}K_{m+1}x_i\subseteq K_m$, which is equivalent to $K_{m+1}\subseteq J_{x_i}$. Consequently, the representatives $k_j$ for $K_m/J_{x_i}$ can be chosen using the finite quotient group $K_m/K_{m+1}$. For the elements $x=k\in K_0$ and $x=x_n^{-1}$, we have $J_x=K_m$ and thus $d=1$.

Since $F$ and $F'$ are $N$-close, in particular $(m+1)$-close, the ring isomorphism $\mathfrak{o}_{E}/\mathfrak{p}_E^{e(m+1)}\xrightarrow{\sim} \mathfrak{o}_{E'}/\mathfrak{p}_{E'}^{e(m+1)}$ matches the quotient
$K_m/K_{m+1}$ (resp. $J_{x_i}/K_{m+1}$) with $K_m'/K_{m+1}'$ (resp. $J_{x_i'}'/K_{m+1}'$).
Now, we choose matching $k_j'\in K_m'$ and obtain the decomposition 
$$ K_m'x_i'K_m = \coprod_{j=1}^d k_j'x_i'K_m'. $$
Set $b_j = k_jx_i$ and $b_j'=k_j'x_i'$. For $x=k\in K_0$ and $x=x_n^{-1}$, we take $b=x$. In particular, we have 
\begin{equation}\label{finer_double_coset}
b_j^{-1}K_{m+1}b_j\subseteq K_m\,\,\, \text{and}\,\,\,b_j'^{-1}K_{m+1}'b_j'\subseteq K_m'.
\end{equation} 
Since both $K_m$, $K_{m+1}$ (resp. $K_m'$, $K_{m+1}'$) are stable under $\sigma$ (resp. $\sigma'$), we have 
\begin{equation}\label{finer_sigma_double_coset}
\sigma(b_j)^{-1}K_{m+1}\sigma(b_j)\subseteq K_m\,\,\, \text{and}\,\,\,
\sigma'(b_j')^{-1}K_{m+1}'\sigma'(b_j')\subseteq K_m'.
\end{equation} 
For the elements $k\in K_0$ and $x_n^{-1}$, (\ref{finer_double_coset}) and (\ref{finer_sigma_double_coset}) are immediate. Put $y=g^{-1}\sigma(g)$. Choose an $\iota'$-fixed representative $y'$ in the double coset corresponding to $K_{m+1}yK_{m+1}$. Then using Hilbert 90, we have $y'=g'^{-1}\sigma'(g')$ for some $g'\in G'$. The right $\mathcal{H}_{\mathcal{C}}'$-action on $\mathcal{C}[H'\backslash G'/K_m']$ gives 
\begin{equation}\label{Hecke_action_prime}
1_{H'g'K_m'}\cdot t_{x_i'} = \sum_{j=1}^d 1_{H'g'k_j'x_i'K_m'}. 
\end{equation}
We now show that the corresponding terms in (\ref{Hecke_action}) and (\ref{Hecke_action_prime}) match individually. Normalize the Haar measures so that ${\rm vol}(K_{m+1})={\rm vol}(K_{m+1}')=1$. Note that ${\rm Kaz}^E_{m+1,\mathcal{C}}(1_{K_m})=1_{K_m'}$. For each element $h\in \{b, b_j :1\leq j\leq d\}$, consider the function
$$ D_{y,h} = 1_{K_m}\ast 1_{K_{m+1} h^{-1}K_{m+1}}\ast 1_{K_{m+1}yK_{m+1}}\ast 1_{K_{m+1}\sigma(h)K_{m+1}}\ast 1_{K_m}. $$
Note that ${\rm supp}(D_{y,h}) = K_mh^{-1}K_{m+1}yK_{m+1}\sigma(h)K_m$. Using (\ref{finer_double_coset}) and (\ref{finer_sigma_double_coset}), we get that
$$ {\rm supp}(D_{y,h}) = K_m\eta(gh)K_m, $$
where $\eta(a)=a^{-1}\sigma(a)$. Similarly,
$$ {\rm supp}(D_{y',h'}) = K_m'\eta(g'h')K_m'. $$
Also note that the Kazhdan map ${\rm Kaz}^E_{m+1,\mathcal{C}}$ sends each factor in $D_{y,h}$ to the corresponding factor in $D_{y',h'}$. Since ${\rm Kaz}^E_{m+1,\mathcal{C}}$ is a $\mathcal{C}$-algebra map, we have 
$$ {\rm Kaz}^E_{m+1,\mathcal{C}}(D_{y,h})=D_{y',h'}. $$ Since this map permutes the integral basis consisting of double cosets, it also matches the two supports. Projecting to level $m$ yields
$$ \beta_m(K_m\eta(gh)K_m) = K_m'\eta(g'h')K_m'. $$
and hence
$$ \alpha_m(HghK_m)=H'g'h'K_m'. $$
Thus, it follows from (\ref{relative_module_isom}), (\ref{Hecke_action}) and (\ref{Hecke_action_prime}) that $\widetilde{\alpha}_{m,\mathcal{C}}$ is compatible with the Hecke algebra actions. Now, tensoring $\widetilde{\alpha}_{m,C}$ with $R$ over $\mathcal{C}$ via the natural ring homomorphism $\mathcal{C}\rightarrow R$ gives an isomorphism of right modules, relative to ${\rm Kaz}_{m,R}$:
$$ \widetilde{\alpha}_{m,R}:R[H\backslash G/K_m]\xrightarrow{\sim} R[H'\backslash G'/K_m']. $$
Hence, we get the following isomorphism of $R$-modules
\begin{equation}\label{coinv_close}
V_H \simeq R[H\backslash G/K_m]\otimes_{\mathcal{H}_R}V^{K_m}\simeq
R[H'\backslash G'/K_m']\otimes_{\mathcal{H}_R'}V'^{K_m} \simeq
V'_{H'}.
\end{equation}
For any ring homomorphism $R\rightarrow S$, we have 
\begin{equation}\label{coinv_comp}
(V\otimes_R S)_H \simeq V_H\otimes_R S.
\end{equation}
Since the averaging operator with respect to $K_m$ is an $R$-linear idempotent, we have 
\begin{equation}\label{fixed_space_comp}
(V\otimes_R S)^{K_m}\simeq V^{K_m}\otimes_R S
\end{equation}
Since ${\rm Kaz}_{m,R}$ and $\widetilde{\alpha}_{m,R}$ are compatible with change of coefficient ring $R\rightarrow S$, it follows from (\ref{coinv_comp}) and (\ref{fixed_space_comp}) that the isomorphism (\ref{coinv_close}) is also compatible with $R\rightarrow S$. 
In particular, (\ref{coinv_close}) implies that $V$ is distinguished by $H$ if and only if $V'$ is distinguished by $H'$. This completes the proof.
\end{proof}


\bibliographystyle{amsalpha}
\bibliography{kazhdan}	

@article {MR3432266,
    AUTHOR = {Ganapathy, Radhika},
     TITLE = {The local {L}anglands correspondence for {$\rm GSp_4$} over
              local function fields},
   JOURNAL = {Amer. J. Math.},
  FJOURNAL = {American Journal of Mathematics},
    VOLUME = {137},
      YEAR = {2015},
    NUMBER = {6},
     PAGES = {1441--1534},
      ISSN = {0002-9327},
   MRCLASS = {11F66 (11F80 22E50)},
  MRNUMBER = {3432266},
MRREVIEWER = {Michael M. Schein},
       DOI = {10.1353/ajm.2015.0045},
       URL = {https://doi.org/10.1353/ajm.2015.0045},
}

@article {Kaz_split,
    AUTHOR = {Kazhdan, D.},
     TITLE = {Representations of groups over close local fields},
   JOURNAL = {J. Analyse Math.},
  FJOURNAL = {Journal d'Analyse Math\'{e}matique},
    VOLUME = {47},
      YEAR = {1986},
     PAGES = {175--179},
      ISSN = {0021-7670},
   MRCLASS = {22E50},
  MRNUMBER = {874049},
MRREVIEWER = {Ernst-Wilhelm Zink},
       DOI = {10.1007/BF02792537},
       URL = {https://doi.org/10.1007/BF02792537},
}

@article {Kaz_generalreductive,
    AUTHOR = {Ganapathy, Radhika},
     TITLE = {A {H}ecke algebra isomorphism over close local fields},
   JOURNAL = {Pacific J. Math.},
  FJOURNAL = {Pacific Journal of Mathematics},
    VOLUME = {319},
      YEAR = {2022},
    NUMBER = {2},
     PAGES = {307--332},
      ISSN = {0030-8730},
   MRCLASS = {22E50 (11F70)},
  MRNUMBER = {4482718},
       DOI = {10.2140/pjm.2022.319.307},
       URL = {https://doi.org/10.2140/pjm.2022.319.307},
}

@article {cong_parahoric_radhika,
    AUTHOR = {Ganapathy, Radhika},
     TITLE = {Congruences of parahoric group schemes},
   JOURNAL = {Algebra Number Theory},
  FJOURNAL = {Algebra \& Number Theory},
    VOLUME = {13},
      YEAR = {2019},
    NUMBER = {6},
     PAGES = {1475--1499},
      ISSN = {1937-0652},
   MRCLASS = {22E50 (11F70)},
  MRNUMBER = {3994573},
MRREVIEWER = {Hengfei Lu},
       DOI = {10.2140/ant.2019.13.1475},
       URL = {https://doi.org/10.2140/ant.2019.13.1475},
}

@book {Vigneras_modl_book,
    AUTHOR = {Vign\'eras, Marie-France},
     TITLE = {Repr\'esentations {$l$}-modulaires d'un groupe r\'eductif
              {$p$}-adique avec {$l\ne p$}},
    SERIES = {Progress in Mathematics},
    VOLUME = {137},
 PUBLISHER = {Birkh\"auser Boston, Inc., Boston, MA},
      YEAR = {1996},
     PAGES = {xviii and 233},
      ISBN = {0-8176-3929-2},
   MRCLASS = {22E35 (20G25 22-02 22E50)},
  MRNUMBER = {1395151},
MRREVIEWER = {Chon\ Hu\ Cheng},
}

@article{Dhar_2024,
author={Dhar, Sabyasachi},
 title={Compatibility of {K}azhdan and {B}rauer homomorphism}, 
DOI={10.4153/S0008439524000596}, journal={Canadian Mathematical Bulletin}, 
year={2024}, 
pages={1–25}
}

@article {Chai_Kang,
    AUTHOR = {Chai, Ching-Li and Yu, Jiu-Kang},
     TITLE = {Congruences of {N}\'eron models for tori and the {A}rtin
              conductor},
      NOTE = {With an appendix by Ehud de Shalit},
   JOURNAL = {Ann. of Math. (2)},
  FJOURNAL = {Annals of Mathematics. Second Series},
    VOLUME = {154},
      YEAR = {2001},
    NUMBER = {2},
     PAGES = {347--382},
      ISSN = {0003-486X,1939-8980},
   MRCLASS = {11S25 (14K15)},
  MRNUMBER = {1865974},
MRREVIEWER = {Antoine\ Chambert-Loir},
       DOI = {10.2307/3062100},
       URL = {https://doi.org/10.2307/3062100},
}

@article {Bruhat_Tits_II,
    AUTHOR = {Bruhat, F. and Tits, J.},
     TITLE = {Groupes r\'eductifs sur un corps local. {II}. {S}ch\'emas en
              groupes. {E}xistence d'une donn\'ee radicielle valu\'ee},
   JOURNAL = {Inst. Hautes \'Etudes Sci. Publ. Math.},
  FJOURNAL = {Institut des Hautes \'Etudes Scientifiques. Publications
              Math\'ematiques},
    NUMBER = {60},
      YEAR = {1984},
     PAGES = {197--376},
      ISSN = {0073-8301,1618-1913},
   MRCLASS = {20G25 (14L15)},
  MRNUMBER = {756316},
MRREVIEWER = {James\ E.\ Humphreys},
       URL = {http://www.numdam.org/item?id=PMIHES_1984__60__5_0},
}

@book {Kaletha_Prasad,
    AUTHOR = {Kaletha, Tasho and Prasad, Gopal},
     TITLE = {Bruhat-{T}its theory---a new approach},
    SERIES = {New Mathematical Monographs},
    VOLUME = {44},
 PUBLISHER = {Cambridge University Press, Cambridge},
      YEAR = {2023},
     PAGES = {xxx+718},
      ISBN = {978-1-108-83196-3},
   MRCLASS = {20E42 (11F70 20G25 22E50)},
  MRNUMBER = {4520154},
}

@incollection {Bernstein_center,
    AUTHOR = {Bernstein, J. N.},
     TITLE = {Le ``centre'' de {B}ernstein},
 BOOKTITLE = {Representations of reductive groups over a local field},
    SERIES = {Travaux en Cours},
     PAGES = {1--32},
    EDITOR = {Deligne, P.},
 PUBLISHER = {Hermann, Paris},
      YEAR = {1984},
      ISBN = {2-7056-5989-7},
   MRCLASS = {22E50},
  MRNUMBER = {771671},
MRREVIEWER = {Fran\c{c}ois\ Rodier},
}

@incollection {Cartier_p_adic_groups,
    AUTHOR = {Cartier, P.},
     TITLE = {Representations of {$p$}-adic groups: a survey},
 BOOKTITLE = {Automorphic forms, representations and {$L$}-functions
              ({P}roc. {S}ympos. {P}ure {M}ath., {O}regon {S}tate {U}niv.,
              {C}orvallis, {O}re., 1977), {P}art 1},
    SERIES = {Proc. Sympos. Pure Math.},
    VOLUME = {XXXIII},
     PAGES = {111--155},
 PUBLISHER = {Amer. Math. Soc., Providence, RI},
      YEAR = {1979},
      ISBN = {0-8218-1435-4},
   MRCLASS = {22E50},
  MRNUMBER = {546593},
MRREVIEWER = {Allan\ J.\ Silberger},
}

@article {Hecke_finiteness,
    AUTHOR = {Dat, Jean-Fran\c cois and Helm, David and Kurinczuk, Robert
              and Moss, Gilbert},
     TITLE = {Finiteness for {H}ecke algebras of {$p$}-adic groups},
   JOURNAL = {J. Amer. Math. Soc.},
  FJOURNAL = {Journal of the American Mathematical Society},
    VOLUME = {37},
      YEAR = {2024},
    NUMBER = {3},
     PAGES = {929--949},
      ISSN = {0894-0347,1088-6834},
   MRCLASS = {22E50 (11F70 11F80)},
  MRNUMBER = {4736530},
       DOI = {10.1090/jams/1034},
       URL = {https://doi.org/10.1090/jams/1034},
}

@article{li2024close,
  title={On close fields and the local {L}anglands correspondence},
  author={Li-Huerta, Siyan Daniel},
  journal={arXiv preprint arXiv:2407.07063},
  year={2024}
}

@article {Solleveld_Meyer_BT,
    AUTHOR = {Meyer, Ralf and Solleveld, Maarten},
     TITLE = {Resolutions for representations of reductive {$p$}-adic groups
              via their buildings},
   JOURNAL = {J. Reine Angew. Math.},
  FJOURNAL = {Journal f\"ur die Reine und Angewandte Mathematik. [Crelle's
              Journal]},
    VOLUME = {647},
      YEAR = {2010},
     PAGES = {115--150},
      ISSN = {0075-4102,1435-5345},
   MRCLASS = {22E50},
  MRNUMBER = {2729360},
MRREVIEWER = {Dubravka\ Ban},
       DOI = {10.1515/CRELLE.2010.075},
       URL = {https://doi.org/10.1515/CRELLE.2010.075},
}

@article {Vincent_GLn_dist,
    AUTHOR = {S\'echerre, Vincent},
     TITLE = {Supercuspidal representations of {${\rm GL}_n(\rm F)$}
              distinguished by a {G}alois involution},
   JOURNAL = {Algebra Number Theory},
  FJOURNAL = {Algebra \& Number Theory},
    VOLUME = {13},
      YEAR = {2019},
    NUMBER = {7},
     PAGES = {1677--1733},
      ISSN = {1937-0652,1944-7833},
   MRCLASS = {22E50 (11F70 11F85)},
  MRNUMBER = {4009674},
MRREVIEWER = {Dubravka\ Ban},
       DOI = {10.2140/ant.2019.13.1677},
       URL = {https://doi.org/10.2140/ant.2019.13.1677},
}

@article {Flicker_dist,
    AUTHOR = {Flicker, Yuval Z.},
     TITLE = {On distinguished representations},
   JOURNAL = {J. Reine Angew. Math.},
  FJOURNAL = {Journal f\"ur die Reine und Angewandte Mathematik. [Crelle's
              Journal]},
    VOLUME = {418},
      YEAR = {1991},
     PAGES = {139--172},
      ISSN = {0075-4102,1435-5345},
   MRCLASS = {22E55 (11F70 22E35 22E50)},
  MRNUMBER = {1111204},
MRREVIEWER = {Dipendra\ Prasad},
       DOI = {10.1515/crll.1991.418.139},
       URL = {https://doi.org/10.1515/crll.1991.418.139},
}

@article {Flicker_Asai,
    AUTHOR = {Flicker, Yuval Z.},
     TITLE = {On zeroes of the twisted tensor {$L$}-function},
   JOURNAL = {Math. Ann.},
  FJOURNAL = {Mathematische Annalen},
    VOLUME = {297},
      YEAR = {1993},
    NUMBER = {2},
     PAGES = {199--219},
      ISSN = {0025-5831,1432-1807},
   MRCLASS = {11F70 (11R39 22E55)},
  MRNUMBER = {1241802},
MRREVIEWER = {Joe\ Repka},
       DOI = {10.1007/BF01459497},
       URL = {https://doi.org/10.1007/BF01459497},
}

@article {Anand_Kable_Asai,
    AUTHOR = {Anandavardhanan, U. K. and Kable, Anthony C. and Tandon, R.},
     TITLE = {Distinguished representations and poles of twisted tensor
              {$L$}-functions},
   JOURNAL = {Proc. Amer. Math. Soc.},
  FJOURNAL = {Proceedings of the American Mathematical Society},
    VOLUME = {132},
      YEAR = {2004},
    NUMBER = {10},
     PAGES = {2875--2883},
      ISSN = {0002-9939,1088-6826},
   MRCLASS = {11F70},
  MRNUMBER = {2063106},
MRREVIEWER = {Jeff\ Hakim},
       DOI = {10.1090/S0002-9939-04-07424-6},
       URL = {https://doi.org/10.1090/S0002-9939-04-07424-6},
}

@article {Nadir_IMRN,
    AUTHOR = {Matringe, Nadir},
     TITLE = {Distinguished generic representations of {${\rm GL}(n)$} over
              {$p$}-adic fields},
   JOURNAL = {Int. Math. Res. Not. IMRN},
  FJOURNAL = {International Mathematics Research Notices. IMRN},
      YEAR = {2011},
    NUMBER = {1},
     PAGES = {74--95},
      ISSN = {1073-7928,1687-0247},
   MRCLASS = {22E50},
  MRNUMBER = {2755483},
MRREVIEWER = {Shu-Yen\ Pan},
       DOI = {10.1093/imrn/rnq058},
       URL = {https://doi.org/10.1093/imrn/rnq058},
}

@article {Anand_IMRN,
    AUTHOR = {Anandavardhanan, U. K. and Rajan, C. S.},
     TITLE = {Distinguished representations, base change, and reducibility
              for unitary groups},
   JOURNAL = {Int. Math. Res. Not.},
  FJOURNAL = {International Mathematics Research Notices},
      YEAR = {2005},
    NUMBER = {14},
     PAGES = {841--854},
      ISSN = {1073-7928,1687-0247},
   MRCLASS = {22E50 (11F70 20G25 22E35)},
  MRNUMBER = {2146859},
MRREVIEWER = {Alain\ Valette},
       DOI = {10.1155/IMRN.2005.841},
       URL = {https://doi.org/10.1155/IMRN.2005.841},
}

@article {Dipendra_distinguished,
    AUTHOR = {Prasad, Dipendra},
     TITLE = {On a conjecture of {J}acquet about distinguished
              representations of {${\rm GL}(n)$}},
   JOURNAL = {Duke Math. J.},
  FJOURNAL = {Duke Mathematical Journal},
    VOLUME = {109},
      YEAR = {2001},
    NUMBER = {1},
     PAGES = {67--78},
      ISSN = {0012-7094,1547-7398},
   MRCLASS = {22E50 (11F70 22E35)},
  MRNUMBER = {1844204},
MRREVIEWER = {Jeff\ Hakim},
       DOI = {10.1215/S0012-7094-01-10912-5},
       URL = {https://doi.org/10.1215/S0012-7094-01-10912-5},
}

@article {mod_l_distinguished,
    AUTHOR = {Kurinczuk, Robert and Matringe, Nadir and S\'echerre, Vincent},
     TITLE = {Cuspidal {$\ell$}-modular representations of {${\rm GL}_n(F)$}
              distinguished by a {G}alois involution},
   JOURNAL = {Forum Math. Sigma},
  FJOURNAL = {Forum of Mathematics. Sigma},
    VOLUME = {13},
      YEAR = {2025},
     PAGES = {Paper No. e48, 41},
      ISSN = {2050-5094},
   MRCLASS = {22E50 (11F70 11F85)},
  MRNUMBER = {4867166},
       DOI = {10.1017/fms.2024.140},
       URL = {https://doi.org/10.1017/fms.2024.140},
}

@article {Anand_IMRN_Asai,
    AUTHOR = {Anandavardhanan, U. K.},
     TITLE = {Root numbers of {A}sai {$L$}-functions},
   JOURNAL = {Int. Math. Res. Not. IMRN},
  FJOURNAL = {International Mathematics Research Notices. IMRN},
      YEAR = {2008},
     PAGES = {Art. ID rnn125, 25},
      ISSN = {1073-7928,1687-0247},
   MRCLASS = {11F70 (11F67 22E50)},
  MRNUMBER = {2448081},
MRREVIEWER = {Fabio\ Mainardi},
       DOI = {10.1093/imrn/rnn125},
       URL = {https://doi.org/10.1093/imrn/rnn125},
}

@article {Flicker_BSMF_1,
    AUTHOR = {Flicker, Yuval Z.},
     TITLE = {Twisted tensors and {E}uler products},
   JOURNAL = {Bull. Soc. Math. France},
  FJOURNAL = {Bulletin de la Soci\'et\'e{} Math\'ematique de France},
    VOLUME = {116},
      YEAR = {1988},
    NUMBER = {3},
     PAGES = {295--313},
      ISSN = {0037-9484},
   MRCLASS = {11F70 (11F55 22E55)},
  MRNUMBER = {984899},
MRREVIEWER = {Daniel\ Bump},
       URL = {http://www.numdam.org/item?id=BSMF_1988__116_3_295_0},
}

@article {Nadir_Kurinczuk_mod_l_Asai,
    AUTHOR = {Kurinczuk, Robert and Matringe, Nadir},
     TITLE = {Characterisation of the poles of the {$\ell$}-modular {A}sai
              {${\rm L}$}-factor},
   JOURNAL = {Bull. Soc. Math. France},
  FJOURNAL = {Bulletin de la Soci\'et\'e{} Math\'ematique de France},
    VOLUME = {148},
      YEAR = {2020},
    NUMBER = {3},
     PAGES = {481--514},
      ISSN = {0037-9484,2102-622X},
   MRCLASS = {22E50 (11F66 11F70)},
  MRNUMBER = {4161168},
MRREVIEWER = {Hidenori\ Katsurada},
       DOI = {10.24033/bsmf.2814},
       URL = {https://doi.org/10.24033/bsmf.2814},
}

@article{dhar2026asai,
  title={Asai Gamma Factors and Distinction in families},
  author={Dhar, Sabyasachi and Sharma, Hariom},
  journal={arXiv preprint arXiv:2603.02699},
  year={2026}
}

@article {Radhika_Pacific_Erratum,
    AUTHOR = {Ganapathy, Radhika},
     TITLE = {A {H}ecke algebra isomorphism over close local fields},
   JOURNAL = {Pacific J. Math.},
  FJOURNAL = {Pacific Journal of Mathematics},
    VOLUME = {319},
      YEAR = {2022},
    NUMBER = {2},
     PAGES = {307--332},
      ISSN = {0030-8730,1945-5844},
   MRCLASS = {22E50 (11F70)},
  MRNUMBER = {4482718},
MRREVIEWER = {Jiandi\ Zou},
       DOI = {10.2140/pjm.2022.319.307},
       URL = {https://doi.org/10.2140/pjm.2022.319.307},
}

\noindent
\texttt{Email:mathsabya93@gmail.com},\\
Department of Mathematics,\\
Indian Institute of Technology Bombay,
Powai, Maharashtra 400076, India.
	
\end{document}